\documentclass[12pt,a4paper,reqno]{amsart}
\usepackage[all]{xy} 
\usepackage{amsmath}
\usepackage{amsfonts}
\usepackage{amssymb}
\usepackage{mathrsfs}

\usepackage{tensor}
\usepackage{cancel}

\usepackage{xcolor}
\usepackage{graphicx}
\newcommand{\Bmu}{\mbox{$\raisebox{-0.59ex}
  {$l$}\hspace{-0.18em}\mu\hspace{-0.88em}\raisebox{-0.98ex}{\scalebox{2}
  {$\color{white}.$}}\hspace{-0.416em}\raisebox{+0.88ex}
  {$\color{white}.$}\hspace{0.46em}$}{}}

\usepackage{amsmath,calligra,mathrsfs}

\newcommand{\scHom}{\mathscr{H}\text{\kern -3pt {\calligra\large om}}}
\newcommand{\scExt}{\mathscr{E}\text{\kern -3pt {\calligra\large xt}}}

\DeclareMathAlphabet{\mathbbold}{U}{bbold}{m}{n}

\DeclareFontEncoding{OT2}{}{} 
\newcommand{\textcyr}[1]{%
 {\fontencoding{OT2}\fontfamily{wncyr}\fontseries{m}\fontshape{n}
 \selectfont #1}}
\newcommand{\Sha}{{\!\be\lbe\mbox{\textcyr{Sh}}}}

\theoremstyle{plain}
\newtheorem{theorem}{Theorem}[subsection]
\newtheorem*{theorema}{Theorem A}
\newtheorem*{theoremb}{Theorem B}
\newtheorem*{theoremc}{Theorem C}
\newtheorem*{corollarya}{Corollary A}
\newtheorem{remark}[theorem]{{\textrm{Remark}}}
\newtheorem{lemma}[theorem]{Lemma}
\newtheorem{corollary}[theorem]{Corollary}

\def\le{\kern 0.03em}

\def\a{\mathfrak{a}}

\def\F{{\mathbb F}}

\def\N{{\mathcal N}}
\def\O{{\mathcal O}}

\def\Q{{\mathbb Q}}
\def\R{{\mathbb R}}

\def\Z{{\mathbb Z}}

\def\e{\kern 0.08em}
\def\be{\kern -.1em}
\def\lbe{\kern -.025em}

\DeclareMathOperator{\coh}{H}

\DeclareMathOperator{\Gal}{Gal}

 \DeclareMathOperator{\Nm}{N}
 \DeclareMathOperator{\image}{Im}
\DeclareMathOperator{\Sel}{{\rm{Sel}}}

\DeclareMathOperator{\Hom}{Hom} 
 
\DeclareMathOperator{\ord}{ord} \DeclareMathOperator{\coker}{coker}

\begin{document}
\title[The $\mu$-invariant change]{The $\mu$-invariant change for abelian varieties over finite $p$-extensions of global fields}
\author{Ki-Seng Tan}
\address{Department of Mathematics\\
National Taiwan University\\
Taipei 10764, Taiwan}
\email{tan@math.ntu.edu.tw}

\author{Fabien Trihan}
\address{Sophia University,
Department of Information and Communication Sciences
7-1 Kioicho, Chiyoda-ku, Tokyo 102-8554, JAPAN}
\email{f-trihan-52m@sophia.ac.jp}

\author{Kwok-Wing Tsoi}
\address{Department of Mathematics\\
National Taiwan University\\
Taipei 10764, Taiwan}
\email{kwokwingtsoi@ntu.edu.tw}

\begin{abstract} We extend the work of \cite{lst21} and study the change of $\mu$-invariants, with respect to a finite Galois p-extension $K'/K$, of an ordinary abelian variety $A$ over a $\Z_p^d$-extension of global fields $L/K$ that ramifies at a finite number of places at which $A$ has ordinary reductions. In characteristic $p>0$, we obtain an explicit bound for the size $\delta_v$ of the local Galois cohomology of the Mordell-Weil group of $A$ with respect to a $p$-extension ramified at a supersingular place $v$. Next, in all characteristics, we describe the asymptotic growth of $\delta_v$ along a multiple $\mathbb{Z}_p$-extension $L/K$ and provide a lower bound for {the change of $\mu$-invariants of $A$ from the tower $L/K$ to the tower $LK'/K'$.} Finally, we present numerical evidence supporting these results.


\end{abstract}
\maketitle

\section{Introduction}\label{s:int} Consider an ordinary abelian variety $A$ defined over a global field $K$. 
If $K$ is a global function field, we let $p$ denote its characteristic; whereas if $K$ is a number field, we denote $p$ as an arbitrary prime number.
Then we fix a $\Z_p^d$-extension $L/K$ unramified outside a finite set of places where $A$ has ordinary (good ordinary, or multiplicative) reduction and set $\Gamma:=\Gal(L/K)$ and $\Lambda_\Gamma=\Z_p[[\Gamma]]$.

In this setting, the Pontryagin dual $X_{L}$ of the $p^\infty$-Selmer group of $A/L$ is known to be
finitely generated over the Iwasawa algebra $\Lambda_\Gamma$ (see Corollary \ref{c:cofin} and 
\cite[Theorem 5]{tan10}). If $X_L$ is torsion, we denote its Iwasawa $\mu$-invariant by $\mu_{L/K}$. By convention, we write $\mu_{L/K}=\infty$ if $X_L$ is not torsion. 

It is noteworthy to mention that the $\lambda$-invariant  of $X_L$ (and hence the `prime-to-$p$ part' of $X_L$,
see \S\ref{su:not}) is invariant under
isogeny and in contrast its $\mu$-invariant (hence its $p$-part) turns out to be more delicate and can change under isogenies (see \cite{sch87}).

To be more specific, for a degree $p$ cyclic extension $K'/K$, if we write $L'=K'L$ and $G=\Gal(K'/K)$, then it can be shown that $L'/K'$ is also a $\Z_p^d$-extension that ramifies
at the ordinary places (see \S\ref{su:not}) only. As a result, $\mu_{L'/K'}$ is defined. Moreover, it turns out that one always has $\mu_{L'/K'}\geq \mu_{L/K}$ (see \eqref{e:mumu'}).
The main goal of this article is to give explicit bounds on the difference of these $\mu$-invariants.

\subsection{The $\mu$-invariants}\label{su:int}  
 
In general, the $\mu$-invariant is difficult to compute explicitly. If $K$ is a global function field, $A/K$ is semi-stable  and $L/K$ is the unramified $\Z_p$-extension $K^{(\infty)}_p/K$, an explicit formula for the the $\mu$-invariant is proven in  \cite[Corollary 2.5.1]{lst21}. { An analogous formula is also proved in loc.\,cit.\,for any Jacobian $A$, or any abelian variety $A$ with finite Tate-Shafarevich groups in the intermediate layers of $K_p^{(\infty)}/K$}. However, in all other cases, it is not known, to our knowledge, how to compute the $\mu$-invariant. If ${\rm char}(K)=p$, the value of $\mu_{L/K}$ is known to be finite in the case where $L$ contains $K^{(\infty)}_p$ (see \cite{ot09}, \cite[Theorem 2]{tan13})
or in the case where the analytic rank of the abelian variety $A$ vanishes and $L/K$ only ramified at good ordinary places. Indeed, in the latter case, the main theorem of Kato and the second author in \cite{kt03} can be used to deduce the validity of the full version of the Birch and Swinnerton-Dyer conjecture. This, in particular, implies that the Mordell-Weil group, the Tate-Shafarevich group and hence the $p$-Selmer group of $A$ over $K$ are all finite. Then the control theorem of the first author in \cite{tan13} can be used to deduce that $X_L$ is torsion.

In the work \cite{lltt15} of Lai, Longhi, the first and second authors, explicit examples for which $\mu_{L/K}$ is infinite are constructed for a dihedral extension by using a multiplicative non-split place that ramifies in the extension. More recently, in \cite{lst21}, the same authors together with Suzuki study the $\mu$-invariant in the case of an unramified $\Z_p$-extension. To be more precise, they have shown that the $\mu$-invariant along $K^{(\infty)}_p/K$ vanishes if and only if $\Sel_p(A/K^{(\infty)}_p)$ is finite (see \cite[Lemma 2.1.1]{lst21}), which will be  the case, for example, if the abelian variety is potentially constant ordinary or potentially constant supersingular and the Hasse-Weil invariant of the curve with function field $K$ becomes invertible after a base change to a suitable finite Galois extension (see \cite[Theorem 1.8 (1)]{ot09}).

While computing the explicit values of $\mu$-invariants could be challenging, in this article we demonstrate the feasibility of bounding the difference between $\mu_{L/K}$ and $\mu_{L'/K'}$. Intriguingly, it turns out that the pivotal factors that cause the {\it change} in $\mu$-invariants are local in nature, revolving around local cohomology groups and the asymptotic growth of their sizes. We shall begin the investigation with a special case.


\subsection{A special case}\label{su:special}
In this subsection, assume that ${\rm char}(K)=p$ and $A/K$ is an elliptic curve having {\em{supersingular reduction}} at a place $v$ where $K'/K$ is {\em{ramified}}.
Our main result will give explicit bounds on the size of $\coh^1(G_v,A(K'_{v'}))$. In this aspect, there appears to be no available literature. 
To proceed further, we introduce some necessary notation.

Let $A^{(p)}$ denote the Frobenius twist of $A$ (see \S\ref{su:not}).
Define $\Gal(k_w/K_v):=\Phi_w$ where $k_w:=K_v(A_p^{(p)}(\bar K^s_v))$.
It is a finite cyclic extension of 
$K_v$ of order dividing $p-1$ because
the Galois action of $\Phi_w$  
induces an embedding 
$$c_w:\Phi_w\longrightarrow \mathrm{Aut}(A_p^{(p)}(\bar K^s_v))\simeq \F_p^*.$$
We will view $c_w$ as an $\F_p^*$-valued character of $\Phi_w$ and for a $\F_p[\Phi_w]$-module $V$, let $V^{c_w}$ denote its $c_w$-eigenspace.
Put $k'_{w'}:=k_wK'_{v'}$ where $v'$ is a place of $K'$ sitting over $v$ and let $e_w$ denote the ramification index of $k_w/K_v$. Then define
$$f_v:=\frac{1}{p-1} \ord_v\mathrm{Disc}(K'_{v'}/K_v)\mbox{ and }f_w:=\frac{1}{p-1} \ord_w\mathrm{Disc}(k'_{w'}/k_w)$$ 
where $\ord_v$ (resp. $\ord_w$) is the valuation on $\bar K_v$, the algebraic closure of $K_v$, whose value equals $1$ at every prime element of $\O_v$ (resp. $\O_w$) and $\mathrm{Disc}$ denotes the discriminant of a field extension.
By the conductor-discriminant formula, $f_v$ and $f_w$ are respectively the valuations of the conductors of non-trivial characters in $\Hom(\Gal(K'_{v'}/K_v),\Z/p\Z)$ and $\Hom(\Gal(k'_{w'}/k_w),\Z/p\Z)$, so, in particular, they are
integers. Since $K'_{v'}/K_v$ and
$k'_{w'}/k_w$ are wildly ramified, we have that $f_v, f_w\geq 2$.
Put $\lambda_v=f_v-1$ and $\lambda_w=f_w-1$.  
They are related by
$\lambda_w=e_w\lambda_v$ (For a proof, see \eqref{e:fwfv} below).

Suppose $A$ is defined by a minimal Weierstrass equation $\mathcal W(x,y)=0$. Let $P=(a,b)$ be a
non-zero geometric point of $A_p$ and define 
$$n_v:=p(p-1)\ord_v\left(\frac{a}{b}\right),\;n_w:=e_wn_v=p(p-1)\ord_w\left(\frac{a}{b}\right).$$
These are positive integers (See \S\ref{su:kw} for a proof).

For $x\in\R$, let $\lceil x\rceil$ (resp. $\lfloor x \rfloor$) denote the smallest (resp. greatest) integer greater (resp. smaller) than or equal $x$. Define $\natural_w:=\frac{pn_w}{p-1}$, which is an integer (see \eqref{e:env}),
put
\begin{equation}\label{e:flat}
\flat_w=\begin{cases}
\lceil \natural_w-\frac{(p-1)\lambda_w}{p}\rceil, & \text{if}\; \lambda_w<\frac{p(\natural_w-1)}{p-1}
;\\
1, &  \mbox{ otherwise,}
\end{cases}
\end{equation}

\begin{equation}\label{e:sharp}
\sharp_w=\begin{cases}
\lceil \natural_w+\frac{(p-1)\lambda_w}{p}\rceil, & \text{if}\; \lambda_w<p\natural_w
;\\
p\natural_w, &  \mbox{ otherwise.}
\end{cases}
\end{equation} 

For each natural number $m$, define
$\varphi_m:=\left\lfloor \dfrac{m-\lceil m/p \rceil}{e_w}\right\rfloor$, $\psi_m:=\lceil m/p \rceil +\varphi_me_w\leq m$,
and let $\diamondsuit_m$ denote the number of
integers belonging to the interval $[\psi_m,m)$ that are congruent to the integer $\frac{n_w}{p-1}$ modulo $e_w$.  Since the interval has width strictly smaller than $e_w$, the value of
$\diamondsuit_m$ is $0$ or $1$. If $\lambda_w\geq p\natural_w$ and $c_w$ is trivial (that is, when $k_w=K_v$), put $\epsilon=1$, ; otherwise, let
$\epsilon=0$. Finally, define
$$\underline{\mathrm b}_v:=\varphi_{\natural_w}+\diamondsuit_{\natural_w}-\varphi_{\flat_w}-\diamondsuit_{\flat_w},\;\;\; \overline{\mathrm b}_v:=\varphi_{\sharp_w}+\diamondsuit_{\sharp_w}-\varphi_{\flat_w}-\diamondsuit_{\flat_w}.$$

\begin{theorema}\label{t:a} Suppose $K_v$ is of characteristic $p$, $A/K_v$ is an elliptic curve having supersingular reduction, and $K'_{v'}/K_v$ is ramified. Under the above notation, 
\begin{equation}\label{e:ta}
 [\F_v:\F_p]\cdot \overline{\mathrm b}_v+\epsilon \geq \log_p|\coh^1(G_v,A(K'_{v'}))|\geq [\F_v:\F_p]
 \cdot \underline{\mathrm b}_v.
\end{equation}
We have $n_v\geq \underline{\mathrm b}_v\geq 0$. Moreover,
\begin{enumerate}
\item[(i)] if $\flat_w=1$, then $\underline{\mathrm b}_v=n_v$; 
\item[(ii)] $\lambda_v=1$ if and only if $\overline{\mathrm b}_v=0$; 
\item[(iii)] if $\lambda_w\geq p\natural_w$, then $\overline{\mathrm b}_v=pn_v$. 
\end{enumerate}
\end{theorema}
\begin{remark}
One of our primary motivations for establishing highly technical bounds on local cohomology groups, as exemplified above, is to render them exceptionally explicit. This enhanced explicitness facilitates their utilization in various other contexts. For instance, the recent work of Overkamp and Suzuki \cite{os23} effectively harnessed the asymptotic version of \eqref{e:ta} to precisely regulate the kernel of a particular base change morphism within the first cohomology of an abelian variety, as demonstrated in \cite[Proposition 4.4]{os23}.
\end{remark}
\subsection{The $\delta$-invariant}\label{su:delta}
Now we return to the general situation as outlined at the beginning of \S\ref{s:int}.
Let $K^{(n)}$ (resp. $K'^{(n)}$) denote the $n$-th layer of $L/K$ (resp. $L'/K'$).
 For a place $u$ of $K^{(n)}$, define
$$\mathfrak H_{u,n}:=\coh^1(G_u, A(K'^{(n)}_{u'}))\subset \coh^1(K_u^{(n)}, A)_p,$$
where we fix a place $u'$ above $u$ (This group is independent of the choice of $u'$).
Put 
$$\mathfrak H_{v}^{(n)}:=\displaystyle\bigoplus_{\text{all}\; u\mid v} \mathfrak H_{u,n}.$$
In general, $\mathfrak{H}_{u,n}$ is finite (see Lemma \ref{l:F'w'}), and so is $\mathfrak{H}_v^{(n)}$, the order of which, along with its asymptotic growth, will be of interest to us.

Consider the case when $\Gamma_v=0$. As there are $p^{nd}$ places sitting over $v$, pick $u$ to be one of them and put 
$\mathfrak H_{u,n}=\coh^1(G_v, A(K'_{v'}))$, we have 
$$\log_p|\mathfrak H_v^{(n)}|=p^{nd}\cdot \log_p|\coh^1(G_v, A(K'_{v'}))|.$$
Theorem A applies to give an estimation of $|\mathfrak H_{u,n}|$ in this setting.

In general, a lower bound of $|\coh^1(G_v, A(K'_{v'}))|$ can be obtained if
$A$ has good reductions. We may assume that $v$ is ramified over $K'/K$, since $\mathfrak H_{u,n}=0$ for unramified $v$ (see Lemma \ref{l:unrambound}). Let $A^t$ denote the dual abelian variety of $A$. Let $\hat A^t$ and $\bar A^t$
respectively denote the associated formal group and the reduction at $v$. Note that the local duality theorem of Tate can be used to imply that $\coh^1(G_v, A(K'_{v'}))$ is dual to $A^t(K_v)/\Nm_{K'_{v'}/K_v}(A^t(K'_{v'}))$
(see, for instance, \cite[Corollary 2.3.3]{tan10}).  On the other hand, the sequence $\xymatrix{0\ar[r] & \hat A^t \ar[r] & A^t \ar[r] & \bar A^t \ar[r] & 0}$ induces an exact sequence
\begin{equation*}
\xymatrix{\hat A^t(\O_v)/\Nm_{K'_{v'}/K_v}(\hat A^t(\O_{v'})) \ar[r] & A^t(K_v)/\Nm_{K'_{v'}/K_v}(A^t(K'_{v'})) 
\ar@{->>}[r] &\bar A^t(\F_v)/\Nm_{K'_{v'}/K_v}(\bar A^t(\F_{v'})),}
\end{equation*}
because, by Hensel's lemma, the reduction $A^t(K_v)\longrightarrow \bar A^t(\F_v)$ is surjective. As a result, we have
\begin{equation*}\label{e:deltalb}
|\coh^1(G_v, A(K'_{v'}))|\geq |\bar A^t(\F_v)/\Nm_{K'_{v'}/K_v}(\bar A^t(\F_{v'}))|=
|\bar A^t(\F_v)/p\bar A^t(\F_v)|=|\bar A(\F_v)/p\bar A(\F_v)|.
\end{equation*}
Interestingly, if $A$ has ordinary reduction at $v$, then
$$|\hat A^t(\O_v)/\Nm_{K'_{v'}/K_v}(\hat A^t(\O_{v'}))|= |\bar A(\F_v)/p\bar A(\F_v)|,$$
(see the proof of Lemma \ref{l:fin}, \cite[Corollary 4.30]{maz} for the number field case) so
the above exact sequence induces
$$2\log|\bar A(\F_v)/p\bar A(\F_v)|\geq \log|\coh^1(G_v, A(K'_{v'}))|\geq \log_p|\bar A(\F_v)/p\bar A(\F_v)|$$
which is a bound similar to that of Theorem A.

It is possible that $\Gamma_v\not=0$ for all non-archimedean places. This is the case, for example, if $L$ contains the cyclotomic $\Z_p$-extension (in characteristic $0$), or the unramified $\Z_p$-extension $K_p^{(\infty)}$.
If $\Gamma_v\not=0$, Theorem A implies the following.
\begin{corollarya}\label{c:a}
Let the assumption and the notation be as in Theorem A. If $\Gamma_v\not=0$, then
$$ p^{dn}\cdot [\F_v:\F_p]\cdot \overline{\mathrm b}_v+\mathrm{O}(p^{n(d-1)})\geq\log_p|\mathfrak H_{v}^{(n)}|\geq p^{dn}\cdot [\F_v:\F_p]\cdot \underline{\mathrm b}_v.$$
\end{corollarya}
\begin{proof} 
Let $w$ be a place of $k$ sitting over $v$.
Because $Lk/k$ is unramified at $w$, for every place $w_n$ of $K^{(n)}k$ sitting over $w$, 
we have $f_{w_n}=f_w$, $e_{w_n}=e_w$, $n_{w_n}=n_w$, and hence for $v_n$ of $K^{(n)}$ sitting below $w_n$,
$\overline{\mathrm b}_{v_n}=\overline{\mathrm b}_v$, $\underline{\mathrm b}_{v_n}=\underline{\mathrm b}_v$. Then check $\sum_{v_n\mid v} [\F_{v_n}:\F_p]=p^{nd}\cdot [\F_v:\F_p]$, where $v_n$ is taken over all places of $K^{(n)}$ sitting over $v$. The number of such $v_n$ is $\mathrm{O}(p^{n(d-1)})$, since $\Gamma_v\not=0$.
\end{proof}
In general, to deal with the problem of the asymptotical growth of $|\mathfrak H_{v}^{(n)}|$, we study
the corresponding Iwasawa Theory.
For each place $v$ of $K$, put (see \S\ref{su:asympt} for details)
$$
\mathfrak H_v:=\varinjlim_n \mathfrak H_v^{(n)}.
$$ 
For a finitely generated $\Lambda_\Gamma$-module $Z$,
 the quotient $Z/pZ$ is torsion and must be pseudo-isomorphic to  to the direct sum of $m$ copies of 
 $\Lambda_\Gamma/p\Lambda_\Gamma$. In the sequel, we refer $m$ to be the `$(p)$-rank' of $Z$. Our next main result describes the asymptotic behaviour of the $(p)$-rank of the Pontryagin dual of  $\mathfrak H_v$.
\begin{theoremb}\label{t:b} For each $v$, the $\Lambda_\Gamma$-module $\mathfrak H_v^\vee$ is
finitely generated and torsion of $(p)$-rank $\delta_v$, which satisfies,
$$\log_p|\mathfrak H_v^{(n)}|=p^{nd}\cdot \delta_v+\mathrm{O}(p^{n(d-1)}).$$
\end{theoremb}

Theorem B will be proven in \S\ref{ss:pfp1}. In the context of Theorem A, through comparison with Corollary A, assuming $\Gamma_v\not=0$, we deduce the following estimate :
\begin{equation}\label{e:a}
[\F_v:\F_p]\cdot \overline{\mathrm b}_v \geq \delta_v \geq  [\F_v:\F_p]\cdot \underline{\mathrm b}_v.
 \end{equation}
 
It is known (see Lemma \ref{l:deltav}) that $\delta_v=0$ if $A$ has ordinary reduction and $\Gamma_v\neq0$, or if $A$ has good reduction and $K'/K$ is unramified. In the latter case, $\mathfrak{H}_{u,n}=0$ for all $n$ and all $u$ lying over $v$, so the group
$$\mathfrak H_n:=\displaystyle\bigoplus_{\text{all}\; w} \mathfrak H_{w,n}$$
is finite. Its order exhibits the asymptotic growth :
\begin{equation}\label{e:asympt}
\log_p|\mathfrak H_n|=p^{nd}\cdot \delta+\mathrm{O}(p^{n(d-1)}),
\end{equation}
where
\begin{equation}\label{e:deltasum}
\delta:=\sum_v\delta_v.
\end{equation}

\subsection{The change of $\mu$-invariants}\label{su:change}
In the setting presented in \cite{lst21}, where $L/K$ represents the unramified $\mathbb{Z}_p$-extension over a global function field $K$, the $\mu$-invariant can be effectively described in terms of the Ulmer's Tate-Shafarevich dimension. However, no analogous descriptions have been established in cases of ramified extensions. Consequently, it becomes crucial for us to undertake a more in-depth examination of the variations in the $\mu$-invariant. This deeper understanding may be instrumental in shedding light on this relationship.

It is already known that if $K'\subset L$, then $\mu_{L/K'}=p\cdot \mu_{L/K}$ 
(this is basically a property of Iwasawa-modules, see \cite[Corollary 1.5, \S1.2(3)]{how02} or \cite[(4)]{cosu05}). Thus,
 for the rest of this article, we assume that the field extensions $K'/K$ and $L/K$ are disjoint and by doing so we 
 identify $\Gal(L'/K')$ with $\Gamma$, $\Gal(L'/L)$ with $\Gal(K'/K)=:G$.  
 
{Let $m_L$ and $m_{L'}$ denote respectively the $(p)$-ranks of $\Sel_{p^\infty}(A/L)^\vee$ and $\Sel_{p^\infty}(A/L')^\vee$.}
Let $\dag$ denote the $(p)$-rank of the Pontryagin dual of $\Sel_{p^\infty}(A/L')^G$.

\begin{theoremc}\label{t:c} Let $A/K$ be an ordinary abelian variety over a global field.
Let $L/K$ be a $\Z_p^d$-extension unramified outside a finite set of ordinary places of $A/K$, disjoint from a cyclic extension $K'/K$ of degree $p$. Let the notation be as the above.  We have
$$\delta+m_L\geq \dag\geq m_L,$$ 
$$(p-1)\dag+m_L \geq m_{L'}\geq \dag .$$
and
$$\mu_{L'/K'}\geq \mu_{L/K}+\dag,$$
Furthermore, if $\mu_{L'/K'}$ is finite, then $\dag \geq \delta$.
\end{theoremc}

Theorem C will be proven in \S\ref{su:pftb}. The proof relies on control lemmas that establish connections between the data $\mu_{L/K}$, $\mu_{L'/K'}$, $m_{L}$, $m_{L'}$, $\delta$, and $\dag$ to asymptotic formulas of orders of various $G$-modules. In general, the data $\mu_{L/K}$, $m_{L}$, and $\delta$ alone are insufficient for deriving an upper bound for $\mu_{L'/K'}$, nor are they adequate for determining if $\mu_{L'/K'}$ is infinite.

\begin{remark}In particular, Theorem C implies that if $\mu_{L/K}=0$, so that $m_L=0$, then $\mu_{L'/K'}\geq \delta$; if in addition $\delta=0$, then $m_{L'}=0$ and hence $X_{L'}$ is torsion with
$\mu_{L'/K'}=0$. This, in view of Lemma \ref{l:deltav}, generalizes the results on the change of $\mu$-invariants 
in \cite[Theorem 3.1]{hm99} and \cite[Theorem 2.1.5]{lst21}.
On the other hand, if $\mu_{L/K}>0$, then $\mu_{L'/K'}\geq \mu_{L/K}+m_L>\mu_{L/K}$, and thus if $K^{(n)}/K$ is a Galois extension 
of degree $p^n$, then $\mu_{LK^{(n)}/K^{(n)}}\geq \mu_{L/K}+n\cdot m_L$.
\end{remark}
Numerical examples are presented in \S\ref{s:comp}. In particular, \S\ref{su:non} displays an example of an elliptic curve $A/K$ together with an extension $K'/K$ such that $\mu_{L/K}$ is finite whereas $\mu_{L'/K'}=\infty$.  

In a subsequent work, we would consider the case of  prime-to-$p$ Galois extensions (such examples have been treated in \cite{ulm19}). Finally, if $K’/K$ is purely inseparable of degree $p$, then $K’=K^{1/p}$, the Frobenius induces an isomorphism between the $p$-Selmer group of $A/K’$ and the $p$-Selmer group of $A^{(p)}/K$. When $\dim A=1$, the comparison between  $\mu_{L/K}$ and $\mu_{L’/K’}$ is given in this case by  \cite[Proposition 4, Proposition 5]{tan21}.

 \subsection{Notations}\label{su:not} For an abelian group $\mathrm{H}$ and an integer $m$, let 
 $\mathrm{H}_m$ denote the kernel of multiplication by $m$, let $\mathrm{H}(m)$ denote the
 $m$-primary component.
Let ${\a}^\vee$ denote the Pontryagin dual of $\a$, when it is defined. 
 
For an algebraic extension $F/K$, let $\Sel_{p^\infty}(A/F)$ and $\Sha_{p^{\infty}}(A/F)$ respectively denote the $p$-primary part of the Selmer group and the Tate-Shafarevich group of $A/F$. 
Let $\Sel_{\rm div}(A/F)$ and $\Sha_{\rm div}(A/F)$ denote their $p$-divisible subgroup. Define
$$\overline\Sha(A/F):=\Sha_{p^\infty}(A/F)/\Sha_{\rm div}(A/F)=\Sel_{p^\infty}(A/F)/\Sel_{\rm div}(A/F).$$
For a positive integer $\nu$, write respectively $\Sel_{p^\nu}(A/F)$, $\Sha_{p^\nu}(A/F)$, and
$\overline\Sha_{p^\nu}(A/F)$ for $\Sel_{p^\infty}(A/F)_{p^\nu}$, $\Sha_{p^\infty}(A/F)_{p^\nu}$, and
$\overline\Sha(A/F)_{p^\nu}$.

If $M/F$ is an algebraic extension of global fields and $S$ is a set of places of $F$, let $S(M)$ denote the set of places of $M$ above those in $S$. 
For a given place $v$, let $\O_v$, $\mathfrak m_v=(\pi_v)$, $\F_v$ (or $\F_{K_v}$) and $q_v$ denote respectively
the ring of integers of the completion of $K$ at $v$, its maximal ideal, its residue field
and the order of its residue field.
Let $\ord_v$ be the normalized valuation on $\bar K_v$ such that $\ord_v \pi_v=1$. 
Let $\bar A$ denote the reduction of $A$ at $v$. 

If $M/F$ is a finite extension, we put $\mathrm{Disc}(M/F)$ for its discriminant. Let $\bar K$ and $\bar K^s$ denote the algebraic and separable closure of $K$.
 
For a Galois group $H$ of an abelian extension $\mathcal L/\mathcal K$ of global fields, we write $H_{v}$ (resp. $H_{v}^0$) for the decomposition subgroup (resp. inertia subgroup).

Recall that $\Gamma:=\Gal(L/K)$, $\Lambda_\Gamma:=\Z_p[[\Gamma]]$ and $G:=\Gal(K'/K)$. 
Write $\Gamma_n$ for $\Gamma^{p^n}$ and let $K^{(n)}:=L^{\Gamma_n}$ be the $n$-th layer of $L/K$, so that $\Gal(L/K^{(n)})=\Gamma_n$. Denote $\Gamma^{(n)}:=\Gal(K^{(n)}/K)$.

Let $Z$ be a finitely generated  
$\Lambda_\Gamma$-module. If $Z$ is torsion, there is an exact sequence 
\begin{equation}\label{e:iwasawa}
\xymatrix{0 \ar[r] &  \bigoplus_{i=1}^m \Lambda_\Gamma/(p^{\alpha_i}) \oplus \bigoplus_{j=1}^n \Lambda_\Gamma/(\eta_j^{\beta_j})
\ar[r]  &  Z \ar[r] &  N \ar[r] & 0,}
\end{equation}
where $\eta_1,...,\eta_n\in \Lambda_\Gamma$ are irreducible, relatively prime to $p$ and $N$ is pseudo-null.
The modules $\bigoplus_{i=1}^m \Lambda_\Gamma/(p^{\alpha_i})$ and $\bigoplus_{j=1}^n \Lambda_\Gamma/(\eta_j^{\beta_j})$ are uniquely determined by $Z$. We will call them respectively the $p$ part and  the non-$p$ part of $Z$. Define the elementary $\mu$-invariants and the $\mu$-invariant of $Z$ to be
$p^{\alpha_1}, ..., p^{\alpha_m}$ and $\sum_{i=1}^m \alpha_i$.
Call $m$ the $(p)$-rank of $Z$.  In general, $Z/pZ$ is torsion, annihilated by $p$, so it is
pseudo-isomorphic to $(\Lambda_\Gamma/p\Lambda_\Gamma)^\rho$ for some $\rho$. If $Z$ is torsion, then $\rho=m$, so in general, define $\rho$ to be the $(p)$-rank of $Z/pZ$. Nakayama's lemma says  
the $(p)$-rank of $Z$ is zero, if and only if $Z$ is torsion and having zero $\mu$-invariant, 
If $Z$ is non-torsion, by convention, define the $\mu$-invariant to be $\infty$.


Since the restriction of the Galois action 
$\Gal(L'/K')\longrightarrow \Gamma$
fixes $K'$, its kernel is trivial, whereas its cokernel equals the finite group $\Gal(L\cap K'/K)$, so $L'/K'$ is also a $\Z_p^d$-extension.
Furthermore, if $w'$ is a place of $L'$ such that the inertia subgroup $\Gal(L'_{w'}/K'_{w'})^0\subset \Gal(L'/K')$ is non-trivial, then the inertia subgroup of $\Gamma$ at $w$ for which $w'\mid w$ contains the image of $\Gal(L'_{w'}/K'_{w'})^0$, and therefore is also non-trivial. Hence, $L'/K'$ is unramified outside a finite set containing only the good ordinary or multiplicative places. As a result, $X_{L'}:=\Sel_{p^\infty}(A/L')^\vee$ is finitely generated over the Iwasawa algebra of $L'/K'$. In particular,  the $\mu$-invariant $\mu_{L'/K'}$ 
associated to $X_{L'}$ is defined. 

In this article, the notation $\coh^*(-,-)$ will either denote the {\it flat cohomology} or {\it group cohomology}. It would be clear from the context which cohomology groups are in use.

Finally, if $\mathrm{char}(K)=p$ and $\dim A=1$,
let $A^{(p)}$ denote the Frobenius twist of $A$, the base change of $A$ over the Frobenius substitution $x\mapsto x^p$, and write
$$\mathsf F:A\longrightarrow A^{(p)}\;\; \text{and}\;\; \mathsf V:A^{(p)}\longrightarrow A$$ 
for the Frobenius and the Verschiebung homomorphisms respectively. Let $A_p$ (resp. $A_p^{(p)}$) denote the kernel of the multiplication by $p$ on $A$
(resp. $A^{(p)}$) .
Recall that the group scheme $C_{p}:=\ker \mathsf{F}$ is the maximal connected subgroup scheme of $A_p$ whereas $E_p^{(p)}:=\ker \mathsf{V}$
is the maximal $\acute{\text{e}}$tale subgroup scheme of $A_p^{(p)}$. These groups are Cartier dual to each other and satisfy the exact sequence (see \cite[\S 3]{lsc10})
\begin{equation}\label{e:exactf}
\xymatrix{0\ar[r] & C_{p}\ar[r]^-{\mathsf i} & A_{p}\ar[r]^-{\mathsf F}  & E_p^{(p)}\ar[r]& 0.}
\end{equation}

\subsection{Acknowledgement} The authors wish to acknowledge the following financial supports to make this research project possible : The first author is partially supported by the MOST grant 109-2115-M-002-008-MY2. 
The second author is partially supported by the JSPS grant 21K03186. The third author is partially supported by the MOST grant 108-2811-M-002-563. We would also like to thank the National Center for Theoretical Sciences (NCTS) for supporting a number of meetings of the authors at the National Taiwan University. The authors would like to thank Takashi Suzuki for helpful conversation and for pointing out some typos in the earlier version of the article, and also thank to Asuka Kumon for helping us with MAGMA.


\section{Theorem A}\label{s:theorema}
In this section, we assume that $K_v$ is a complete, characteristic $p$, local field of finite residue field,
$A/K_v$ is an elliptic curve having supersingular reduction, and $K'_{v'}/K_v$ is a ramified
cyclic extension of degree $p$. Write $G_{v}$ for $\Gal(K'_{v'}/K_v)$.

\subsection{The first inequality}\label{su:firstin}
Consider the commutative diagram of exact sequences
$$\xymatrix{ 0\ar[r] & \ker(\mathfrak a_v) \ar[r] \ar[d] & \coh^1(K_v, A_p) \ar[r]^-{\mathfrak a_v}\ar[d]^-{\tilde{\mathfrak a}_v}  & \coh^1(K_v,A)_p\ar[r] \ar[d]^-{\mathrm{res}_v} & 0 \\
 0\ar[r] & 0\ar[r] &  \coh^1(K'_{v'},A)_p \ar[r]^-{\mathrm{id}} &  \coh^1(K'_{v'},A)_p\ar[r] & 0,}
$$
where $\mathfrak a_v$ is obtained from
the Kummer exact sequences and $\mathrm{res}_v$ is the restriction map.
The snake lemma implies 
\begin{equation}\label{e:atildea}
\coh^1(G_v,A(K'_{v'}))=\ker(\mathrm{res}_v)=\ker(\tilde{\mathfrak a}_v)/\ker(\mathfrak a_v).
\end{equation}

Put 
${\mathfrak c}_v:={\mathfrak a}_v\circ \mathsf i_*$, $\tilde{\mathfrak c}_v:=\tilde{\mathfrak a}_v\circ \mathsf i_*$, where $\mathsf i$ is the map in \eqref{e:exactf}.
Then $\ker(\mathsf i_*)\subset \ker(\mathsf c_v)$ and
$$\xymatrix{\ker(\mathfrak c) \ar@{->>}[r]^-{\mathsf i_*} &\ker(\mathfrak a)\cap \image(\mathsf i_*)},
\;\; \xymatrix{\ker(\tilde{\mathfrak c})\ar@{->>}[r]^-{\mathsf i_*} & \ker(\tilde{\mathfrak a})\cap \image(\mathsf i_*)},
$$
so
\begin{equation}\label{e:ovi}
\xymatrix{\ker(\tilde{\mathfrak c})/\ker(\mathfrak c) \ar@{^(->}[r]^-{\overline{\mathsf i}} & \ker(\tilde{\mathfrak a})/\ker(\mathfrak a)}.\end{equation} 

Define $\tilde{\mathfrak e}_v$ to be the composition
$$\xymatrix{ \coh^1(K_v, E^{(p)}_p)  \ar[r]^-{\mathfrak e_v} & 
\coh^1(K_{v},A^{(p)})_p \ar[r]^-{\mathrm{res}^{(p)}_v} & \coh^1(K'_{v'}, A^{(p)})_p},$$
where $\mathfrak e_v$ is induced from the natural map $E^{(p)}_p\longrightarrow A_p^{(p)}$.
The commutative diagram 
$$\xymatrix{A_p \ar[r]^-{\mathsf F} \ar[d] & E_p^{(p)}\ar[d] \\
A\ar[r]^-{\mathsf F} & A^{(p)}}
$$
implies that $\mathsf F_*(\ker(\mathfrak a)) \subset \ker(\mathfrak e)$,  $\mathsf F_*(\ker(\tilde{\mathfrak a})) \subset \ker(\tilde{\mathfrak e})$, hence induces the sequence
\begin{equation}\label{e:ovif}
\xymatrix{0\ar[r] &
\ker(\tilde{\mathfrak c})/\ker(\mathfrak c) \ar[r]^-{\overline{\mathsf i}} & \ker(\tilde{\mathfrak a})/\ker(\mathfrak a)
\ar[r]^-{\overline{\mathsf F}} & \ker(\tilde{\mathfrak e})/\ker(\mathfrak e),}
\end{equation}
which by \eqref{e:ovi}, is actually exact. 
Therefore, by \eqref{e:atildea},

\begin{equation}\label{e:firstin}
|\ker(\tilde{\mathfrak c}_v)/\ker(\mathfrak c_v)|\cdot |\ker(\tilde{\mathfrak e}_v)/\ker(\mathfrak e_v)|\geq 
|\coh^1(G_v,A(K'_{v'})) | \geq |\ker(\tilde{\mathfrak c}_v)/\ker(\mathfrak c_v)|.
\end{equation}
{In the following, we will show that Theorem A simply follows by unwinding the terms in inequality (\ref{e:firstin}). In the course of doing so, we will rely on the recent results from \cite{tan21} of the first-named author.}

\subsection{The field $k_w$}\label{su:kw}
Over the field $k_w$, the group scheme $E_p^{(p)}$ is identified with $\mathbb{Z}/p\mathbb{Z}$, on which $\Phi_w$ acts via $c_w$. Its Cartier dual, denoted by $C_p = \Bmu_p$, is endowed with an action of $\Phi_w$ via $c_w^{-1}$.

Denote $k'_{w'}=k_wK'_{v'}$. Since $[k_w:K_v]$ is prime to $p$, the
field extensions $k_w/K_v$ and $K'_{v'}/K_v$ are disjoint from each other, so
we identify $\Gal(k'_{w'}/k_w)$ with $G_{v'}$ and $\Gal(k'_{w'}/K'_{v'})$ with $\Phi_w$.

If we take $K_v=k_w$ in \S\ref{su:firstin}, we define the homomorphisms: $\mathfrak c_w$, $\tilde{\mathfrak c}_w$,
$\mathfrak a_w$, $\tilde{\mathfrak a}_w$, $\mathfrak e_w$, $\tilde{\mathfrak e}_w$.

Since $[k_w:K_v]$ is relatively prime to $p$, the restriction map identify $\coh^1(K_v,C_p)$ with
$\coh^1(k_w,C_p)^{\Phi_w}$, $\ker(\mathfrak c_v)$ with $\ker(\mathfrak c_w)^{\Phi_w}$, and
$\ker(\tilde{\mathfrak c}_v)$ with $\ker(\tilde{\mathfrak c}_w)^{\Phi_w}$, so
\begin{equation}\label{e:cphiw}
\ker(\tilde{\mathfrak c}_v)/\ker(\mathfrak c_v)=\ker(\tilde{\mathfrak c}_w)^{\Phi_w}/\ker(\mathfrak c_w)^{\Phi_w}=
(\ker(\tilde{\mathfrak c}_w)/\ker(\mathfrak c_w))^{\Phi_w}.
\end{equation}
A similar argument shows that
\begin{equation}\label{e:ephiw}
\ker(\tilde{\mathfrak e}_v)/\ker(\mathfrak e_v)=\ker(\tilde{\mathfrak e}_w)^{\Phi_w}/\ker(\mathfrak e_w)^{\Phi_w}=
(\ker(\tilde{\mathfrak e}_w)/\ker(\mathfrak e_w))^{\Phi_w}.
\end{equation}

Now we choose a minimal Weierstrass equation $\mathcal W(x,y)=0$ that defines $A/K_v$ and choose a non-zero point $P=(a,b)\in A_p(\bar K_v)$. Then
$\mathsf F(P)=(a^p,b^p)$ is a non-zero point in $E_p^{(p)}(\bar K_v^s)$. 
Since the reduction of $P$ is the identity of $\bar A$, we have $t:=-a^p/b^p\in \mathfrak m_w$ (see \cite[VII \S2]{sil86}).
Then we can deduce that $n_v=(p-1)\ord_v(t)$ is a positive integer which coincides with the one in \cite[(9)]{tan21}. 
Moreover, since $e_w$ is the ramification index of $k_w/K_v$, 
\begin{equation}\label{e:env}
n_w= e_wn_v=(p-1) \ord_w t\equiv 0\pmod{p-1}.
\end{equation} 

For a finite extension $\mathsf K_{\mathsf v}/K_v$,
put $n_{\mathsf v}:=(p-1)\ord_{\mathsf v}(t)$ and $\natural_{\mathsf v}=\frac{pn_{\mathsf v}}{p-1}$.
By taking $K_v=\mathsf K_{\mathsf v}$ in \S\ref{su:firstin}, define the homomorphisms: $\mathfrak c_{\mathsf v}$, $\mathfrak a_{\mathsf v}$, and $\mathfrak e_{\mathsf v}$.
Suppose $\mathsf K_{\mathsf v}$ contains $k_w$. Then 
$E_p^{(p)}=\Z/p\Z$ over $\mathsf K_{\mathsf v}$, so, as abelian groups,
\begin{equation}\label{e:cohe}\coh^1(\mathsf K_{\mathsf v}, E_p^{(p)})=\Hom(\Gal(\bar {\mathsf K}_{\mathsf v}^s/\mathsf K_{\mathsf v}),\Z/p\Z)=\Hom(\mathsf K_{\mathsf v}^*/(\mathsf K_{\mathsf v}^*)^p, \Q_p/\Z_p),
\end{equation}
where the second equality follows from the local class field theory.
For a given character $\chi\in \Hom(\mathsf K_{\mathsf v}^*/(\mathsf K_{\mathsf v}^*)^p, \Q_p/\Z_p)$, let the ideal
$(\pi_{\mathsf v})^{f_\chi}\subset \O_{\mathsf v}$ denote its conductor. {Under this identification, we can compute the kernel of ${\mathfrak e_{\mathsf v}}$}


\begin{lemma}\label{l:disc}
Suppose $\mathsf K_{\mathsf v}$ contains $k_w$. The kernel of
$$\xymatrix{\coh^1(\mathsf K_{\mathsf v}, E_p^{(p)}) \ar[r]^-{\mathfrak e_{\mathsf v}} &
\coh^1( \mathsf K_{\mathsf v}, A^{(p)})}$$
consists of the characters $\chi\in \Hom(\mathsf K_{\mathsf v}^*/(\mathsf K_{\mathsf v}^*)^p, \Q_p/\Z_p)$
such that $f_\chi\leq \natural_{\mathsf v}$.
\end{lemma}
\begin{proof}
See \cite[Lemma 2.3.3]{tan21}.
\end{proof}

For each positive integer $n$, put $W_{\mathsf v,n}:=1+\pi_{\mathsf v}^n\O_{\mathsf v}$ 
as a subgroup in $\O_\mathsf v^*$, write $\tilde M_{\mathsf v}$ for
$W_{\mathsf v,\natural_{\mathsf v}}\cdot (\mathsf K_{\mathsf v}^*)^p$
and 
for integers $m>n>0$, define the multiplicative groups
$$W_{\mathsf v,n,m}:=W_{\mathsf v,n}/W_{\mathsf v,m}\mbox{ and }\overline{W}_{\mathsf v,n,m}:=W_{\mathsf v,n,m}/(W_{\mathsf v,n,m}\cap W_{\mathsf v,1,m}^p).$$

\begin{lemma}\label{l:kerev} We have 
$$\ker(\mathfrak e_v)=\Hom(k_w^*/\tilde M_w,\Q_p/\Z_p)^{c_w^{-1}},\;\;
\ker(\tilde{\mathfrak e}_v)={\rm res}_w^{-1}(\Hom(k'^*_{w'}/\tilde M_{w'},\Q_p/\Z_p))^{c_w^{-1}}.$$
\end{lemma}
\begin{proof} {Lemma \ref{l:disc} implies that $\ker(\mathfrak e_w)=\Hom(k_w^*/\tilde M_w,\Q_p/\Z_p)=\Hom(k_w^*/\tilde M_w,\frac{1}{p}\Z_p/\Z_p)$
and that $\ker(\mathfrak e_{w'})=\Hom(k'^*_{w'}/\tilde M_{w'},\Q_p/\Z_p)=\Hom(k'^*_{w'}/\tilde M_{w'},\frac{1}{p}\Z_p/\Z_p)$ as abelian groups.  Also, if 
\begin{equation}\label{e:resw}
{\rm res}_w:\Hom(\Gal(\bar k_w^s/k_w),\frac{1}{p}\Z_p/\Z_p)\longrightarrow \Hom(\Gal(\bar {k'}_w^s/k'_w),
\frac{1}{p}\Z_p/\Z_p)
\end{equation}
is the restriction map induced by the inclusion $\Gal(\bar {k'}_w^s/k'_w)\longrightarrow \Gal(\bar k_w^s/k_w)$,
then
\begin{equation}\label{e:tildee}
\ker(\tilde{\mathfrak e}_w)={\rm res}_w^{-1}(\Hom(k'^*_{w'}/\tilde M_{w'},\Q_p/\Z_p)),
\end{equation}
as abelian groups. To take the $\Phi_w$-action into consideration, we note that as a Galois-module, 
$E_p^{(p)}=V_w$, the $1$-dimensional $\F_p$-vector space such that $\tensor[^\sigma]x{}=c_w(\sigma) x$, for all $\sigma\in\Phi_w$, $x\in V_w$. This implies that, for Galois cohomology groups, we have
\begin{equation}\label{e:smallint}
\coh^1(k_w, E_p^{(p)})=V_w\otimes_{\F_p}\coh^1(k_w,\Z/p\Z)=V_w\otimes_{\F_p}\Hom(k_w^*/(k_w^*)^p, \Q_p/\Z_p),
\end{equation}
as $\Phi_w$-modules. Note that, again, the second equality is due to the local class field theory, such that for 
$\sigma\in\Phi_w$, $\chi\in \Hom(k_w^*/(k_w^*)^p, \Q_p/\Z_p)$, $\tensor[^\sigma]\chi{}$ is the character that sends every $x\in k_w^*/(k_w^*)^p$
to $\chi{}(\tensor[^{\sigma^{-1}}]x{})$, where $\tensor[^{\sigma^{-1}}]x{}$ is induced from the Galois action of $\Phi_w$ on $k_w^*$. This shows
$$\coh^1(K_v,E_p^{(p)})=\coh^1(k_w,E_p^{(p)})^{\Phi_w}=\Hom(k_w^*/(k_w^*)^p, \Q_p/\Z_p)^{c_w^{-1}},$$
and for a $\Phi_w$-submodule $\mathsf H\subset \coh^1(k_w,E_p^{(p)})$ that corresponds 
to $\mathsf Q\subset \Hom(k_{w}^*/(k_{w}^*)^p, \Q_p/\Z_p)$, under \eqref{e:cohe}, we have $\mathsf H^{\Phi_w}=\mathsf Q^{c_w^{-1}}.$}
\end{proof}
\begin{lemma}\label{l:ann} Suppose $\mathsf K_{\mathsf v}$ is a finite extension of $K_v$.
In $\coh^1(\mathsf K_{\mathsf v},C_p)$, which is viewed as the Pontryagin dual of $\coh^1(\mathsf K_{\mathsf v},E_p^{(p)})$, the annihilator of $\ker(\mathfrak e_{\mathsf v})$ is $\ker(\mathfrak c_{\mathsf v})$.
\end{lemma}
\begin{proof}See \cite[Lemma 3.15]{tan21}.
\end{proof}
   
\begin{lemma}\label{l:kercv} We have
$$\ker(\mathfrak c_v)=\ker(\mathfrak c_w)^{\Phi_w}=(W_{w,\natural_w}/(W_{w,\natural_w}\cap W_{w,1}^p))^{c_w}
=((W_{w,\natural_w}\cdot W_{w,1}^p)/W_{w,1}^p)^{c_w},$$
$$\ker(\tilde{\mathfrak c}_v)=\ker((\tilde{\mathfrak c}_w)^{\Phi_w}=((W_{w,1}\cap (W_{w',\natural_{w'}}\cdot  W_{w',1}^p))/W_{w,1}^p)^{c_w},$$
and hence
$$\ker(\tilde{\mathfrak c}_v)/\ker(\mathfrak c_v)=((W_{w,1}\cap (W_{w',\natural_{w'}}\cdot  W_{w',1}^p))/(W_{w,\natural_w}\cdot W_{w,1}^p))^{c_w}.$$ 

\end{lemma}       
\begin{proof}
{Lemma \ref{l:ann} together with the commutative diagram
$$\xymatrix{\ker(\mathfrak e_w) \ar@{^(->}[r] \ar@{=}[d]& \coh^1(k_w,E_p^{(p)})\ar@{=}[d] \\
                    \Hom(k_w^*/\tilde M_w,\Q_p/\Z_p) \ar@{^(->}[r] & \Hom(k_w^*/(k_w^*)^p,\Q_p/\Z_p),}
                    $$
gives rise to the identification $\coh^1(k_w,C_p)=\coh^1(k_w,E_p^{(p)})^\vee=k_w^*/(k_w^*)^p$, and hence
$$\ker(\mathfrak c_w)=\tilde M_w/(k_w^*)^p=W_{w,\natural_w}/W_{w,\natural_w}\cap W_{w,1}^p
,$$                  
as abelian groups. Taking $\Phi_w$-action into consideration, we have
$$\ker(\mathfrak c_v)= (W_{w,\natural_w}/W_{w,\natural_w}\cap W_{w,1}^p)^{c_w}.$$                  
Also, by the commutative diagram
\begin{equation}\label{e:j}
\xymatrix{\coh^1(k_w,C_p) \ar[r] \ar@{=}[d] & \coh^1(k'_{w'},C_p) \ar@{=}[d] & \ker(\mathfrak c_{w'})\ar@{_(->}[l]\ar@{=}[d]\\
k_w^*/(k_w^*)^p \ar[r]^-j & {k'^*_{w'}}/({k'^*_{w'}})^p  & W_{w',\natural_{w'}}/W_{w',\natural_{w'}}\cap W_{w',1}^p\ar@{_(->}[l],}
\end{equation}
where the upper-left arrow is the restriction map  and $j$ is induced from the inclusion 
$k_w^*\subset {k'^*_{w'}}$, we obtain
$$\begin{array}{rcl}
\ker(\tilde{\mathfrak c}_w)&=& j^{-1}( W_{w',\natural_{w'}}/W_{w',\natural_{w'}}\cap W_{w',1}^p)\\
&=&  j^{-1}( W_{w',\natural_{w'}}\cdot W_{w',1}^p/ W_{w',1}^p)\\
&=& (k_w^*\cap (W_{w',\natural_{w'}}\cdot W_{w',1}^p))/W_{w,1}^p\\
&=& (W_{w,1}\cap (W_{w',\natural_{w'}}\cdot  W_{w',1}^p))/W_{w,1}^p
\end{array}.$$}
\end{proof}

\subsection{Change of conductors}\label{su:chcond}
We retain the notation from \S\ref{su:kw}. For a character $\chi\in \Hom(\Gal(\bar k^s_w/k_w), \frac{1}{p}\Z_p/\Z_p)$,
let $k_{w,\chi}$ denote the fixed field of $\ker(\chi)$, and put $\tilde\chi={\rm res}_w(\chi)$, the image under \eqref{e:resw}. The conductor-discriminant formula \cite[Poposition 6, VI]{ser79} says
$$\ord_w\mathrm{Disc}(k_{w,\chi}/k_w)=(p-1)f_{\chi}.$$
We have $k'_{w'}=k_{w,\omega}$ for some $\omega$, so that $f_w=f_\omega$. 
For a given $\chi$, the restriction $\tilde\chi=0$ if and only if $\chi=\omega^a$ for some
integer $a$. For convenience, assume that $\tilde \chi\not=0$, whence
$k'_{w',\tilde\chi}=k'_{w'}k_{w,\chi}$ is an abelian extension of $k_w$ of degree $p^2$.
Because $\ord_{w'}=p\ord_w$, the transitivity of the discriminant \cite[Proposition 8,III]{ser79} implies
\begin{equation}\label{e:disctrans}
\ord_w\mathrm{Disc}(k'_{w',\tilde\chi}/k_w)=p(p-1)f_w+(p-1)f_{\tilde\chi}.
\end{equation}
We discuss the value of $f_{\tilde\chi}$ in two cases.
\subsubsection{The first case}\label{ss:p=>o}Suppose $f_{\chi\omega^b}\geq f_w$ for all $b\in\Z$. 
Then, for $p\nmid a$, we have $f_{\chi^a\omega^b}=f_\chi$. The conductor-discriminant formula says
\begin{equation*}\label{e:discpsi}
\ord_w\mathrm{Disc}(k'_{w',\tilde\chi}/k_w)=(p-1)f_w+p(p-1)f_\chi.
\end{equation*}
This together with \eqref{e:disctrans} implies
\begin{equation}\label{e:p=>o}
f_{\tilde\chi}=pf_\chi-(p-1)f_w.
\end{equation}

\subsubsection{The second case}\label{ss:p<o}
Suppose $f_\chi< f_w$. 
Then, for $p\nmid b$, we have $f_{\chi^a\omega^b}=f_w$ for all $a$. 
In this case, we have
\begin{equation*}
\ord_w\mathrm{Disc}(k'_{w',\tilde\chi}/k_w)=p(p-1)f_w+(p-1)f_\chi
\end{equation*}
so that
\begin{equation}\label{e:p<o}
f_{\tilde\chi}=f_\chi.
\end{equation}

\subsubsection{Multiplicative groups}\label{ss:multg}  
{We remind readers that, for integers $m>n>0$, we have defined the multiplicative groups
$$W_{w,n,m}:=W_{w,n}/W_{w,m}\mbox{ and }\overline{W}_{w,n,m}:=W_{w,n,m}/(W_{w,n,m}\cap W_{w,1,m}^p).$$} Lemma \ref{l:kercv} says
\begin{equation}\label{e:kercvcv}
\ker(\tilde{\mathfrak c}_v)/\ker(\mathfrak c_v)\subset {\overline{W}_{w,1,\natural_w}^{c_w}}.
\end{equation}
Also, according to Lemma \ref{l:kerev},
\begin{equation}\label{e:kerevvee}\begin{array}{rcl}
\ker(\mathfrak e_v)^\vee&=& (k_w^*/\tilde M_w)^{c_w}\\
&=& ((\O_w^*\times\Z)/((W_{w,\natural_w}\cdot (\O_w^*)^p)\times \Z/p\Z))^{c_w}\\
&=& (\overline{W}_{w,1,\natural_w}\times \Z/p\Z)^{c_w}.
\end{array}
\end{equation}

Recall the invariant $\sharp_w$ defined in \S\ref{su:special}. Since $k'_{w'}/k_w$ is totally ramified,  we have 
$\natural_{w'}=p\natural_w$. 
If $\lambda_w<p\natural_{w}$,
then $\sharp_w=\lfloor\frac{1}{p}\natural_{w'}+\frac{p-1}{p}f_w\rfloor\leq \natural_{w'}$.  The following result computes the group $\ker(\tilde{\mathfrak e}_v)^\vee$.
\begin{lemma}\label{l:kertildeevvee} If $\lambda_w\geq p\natural_{w} $, then 
$\ker(\tilde{\mathfrak e}_v)^\vee=(\overline{W}_{w,1,\sharp_{w}}\times \Z/p\Z\times \Z/p\Z)^{c_w}$; otherwise, 
$\ker(\tilde{\mathfrak e}_v)^\vee=(\overline{W}_{w,1,\sharp_w}\times \Z/p\Z)^{c_w}$.
\end{lemma}
\begin{proof} Suppose $\lambda_w\geq\natural_{w'}$ so that $f_w>\natural_{w'}$. By \eqref{e:p<o}, if $f_\chi\leq \natural_{w'}$, then
$f_{\tilde \chi}\leq \natural_{w'}$; if $f_w>f_\chi>\natural_{w'}$, then
$f_{\tilde\chi}>\natural_{w'}$. By \eqref{e:p=>o}, if $f_{\chi\omega^b}\geq f_w$, for all $b$,
then
$f_{\tilde\chi}>\natural_{w'}$. Lemma \ref{l:disc} thus says
$${\rm res}_w^{-1}(\ker(\mathfrak e_{w'}))=\ker({\rm res}_w)\cdot \Hom(k_w^*/W_{w,\natural_{w'}}\cdot (k^*_w)^p,\Q_p/\Z_p)$$
Since $\ker({\rm res}_w)=\langle\omega\rangle=\Z/p\Z$ and $f_\omega> \natural_{w'}$, the first assertion follows from 
a computation similar to that of \eqref{e:kerevvee}.

Suppose $\lambda_w< \natural_{w'}$ and hence $f_w\leq \natural_{w'}$. If $f_\chi<f_\omega$, then $f_{\tilde\chi}=f_\chi<\natural_{w'}$.
If $f_{\chi\omega^b}\geq f_\omega$, for all $b$, then $f_{\tilde\chi}=pf_\chi-(p-1)f_\omega$. In such case, $f_{\tilde\chi}\leq \natural_{w'}$
if and only if $f_\chi\leq \sharp_w$. Since $\ker(res_w)\subset \Hom(k_w^*/W_{w,u}\cdot (k^*_w)^p,\Q_p/\Z_p)$,
the second assertion follows.

\end{proof}

\subsection{$\Z[\Phi_w]$-modules}\label{su:phimodule} In this subsection, we will discuss the structure of 
$W_{w,n,m}$ and $\overline{W}_{w,n,m}$ as $\Z[\Phi_w]$-modules. For simplicity, write $W_{n,m}$ and $\overline W_{n,m}$ for them.

\begin{lemma}\label{l:basic} The assignment $\xi\mapsto 1+\xi$ induces a bijection 
$$\rho_{n,m}: \pi_w^n\O_w/\pi_w^m\O_w\longrightarrow W_{n,m}.$$
that respects $\Phi_w$-actions. If $2n\geq m$, then $\rho_{n,m}$ is a group isomorphism.

\end{lemma}

\begin{proof} For any $a,b\in\O_w$,
\begin{equation}\label{e:basic}
\begin{array}{rcl}
(1+a\pi_w^n)\cdot (1+b\pi_w^n)^{-1} &=& (1+a\pi_w^n)\cdot(1+\sum_{k=1}^\infty (-1)^k b^k \pi_w^{nk})\\
&=& 
1+(b-a)\pi_w^n\cdot \sum_{k=0}^\infty (-1)^{k+1} b^{k}\pi_w^{nk}.
\end{array}
\end{equation}
Therefore, $(1+a\pi_w^n)\cdot (1+b\pi_w^n)^{-1}\in 1+\pi_w^m\O_w$ if and only if $(a-b)\pi_w^n\in\pi_w^{m}\O_w$. This proves the first assertion.
If $2n\geq m$, then
$$(1+a\pi_w^n)(1+b\pi_w^n)=1+(a+b)\pi_w^n+ab\pi_w^{2n}\equiv 1+(a+b)\pi_w^n\pmod{\pi_w^m\O_w},$$
so it determines the same class as $1+(a+b)\pi_w^n$ in $W_{n,m}$.
\end{proof}

By viewing $\F_v$ as a subring of $\O_w$, the assignment $a\pi_w^n\mapsto \alpha\cdot a\pi_w^n$, for $\alpha\in\F_v$,
endows $\pi_w^n\O_w/\pi_w^m\O_w$ with a structure of $\F_v$-vector space, and hence a structure of $\F_v[\Phi_w]$-module.
Thus, if $2n\geq m$, then via $\rho_{n,m}$, we endow $W_{n,m}$ with an $\F_v[\Phi_w]$-module structure. In particular $W_{l,l+1}$
is always an $\F_v[\Phi_w]$-module. 

In general, we can apply the Jordan-H\"{o}lder Theorem. For a $\Z[\Phi_w]$-module $W$ of finite length, let $\mathsf {JH}W$ denote the direct sum of the composition factors (with multiplicities) of $W$. An exact sequence
$\xymatrix{0\ar[r]  & X \ar[r]  & W \ar[r]  & Y\ar[r]  & 0}$ of $\Z[\Phi_w]$-modules gives rise to 
$$\mathsf {JH}W=\mathsf {JH}X\oplus \mathsf {JH}Y.$$
Thus, from the exact sequences, $l=n,...,m-1$,$$\xymatrix{0\ar[r] & W_{l+1,m} \ar[r] & W_{l,m} \ar[r] & W_{l,l+1}  \ar[r] & 0,}$$
we deduce that
\begin{equation}\label{e:decomp}
\mathsf {JH} W_{n,m}=\bigoplus_{l=n}^{m-1} \mathsf {JH}W_{l,l+1}.
\end{equation} 
Note that since the order of $\Phi_w$ divides $p-1$ and all $(p-1)$-th roots of $1$ are in $\F_p$, every finite $\F_v[\Phi_w]$-module 
(resp. $\F_p[\Phi_w]$-module) is a direct sum of $1$-dimensional $\Phi_w$-eigenspaces over $\F_v$ (resp. $\F_p$).
In particular, $W_{l,l+1}$ is a direct sum of $1$-dimensional $\Phi_w$-eigenspaces over $\F_v$,
and each such eigenspace is an $[\F_v:\F_p]$-copies of a $\Phi_w$-eigenspaces over $\F_p$.
Therefore, every composition factor of $W_{n,m}$ is an $1$-dimensional $\Phi_w$-eigenspace over $\F_p$, and hence a group of order $p$.

Similarly, if $2n\geq m$, then $\overline W_{n,m}$ is endowed with a structure of $\F_v[\Phi_w]$-module.
Since $\overline W_{n,m}$ is annihilated by $p$, it has a natural structure of $\F_p[\Phi_w]$-module, it is easy to see that if $2n\geq m$,
then this coincides with the one induced from the $\F_v[\Phi_w]$-module structure. 

For the next result, we let $\Phi_w^0\subset \Phi_w $ denote the inertia subgroup and put 
$$\bar \Phi_w:=\Phi_w/\Phi_w^0=\Gal(\F_w/\F_v).$$


\begin{lemma}\label{l:reg} The following holds.
\begin{enumerate}
\item[(a)] The representation of $\Phi_w^0$ on $W_{1,2}$ is faithful.
\item[(b)] For $n>0$, $\mathsf{JH}W_{n,n+e_w}=\mathsf {JH}\F_v[\Phi_w]$. 
\item[(c)] For $m>n$,
$\mathsf{JH}W_{n+e_w,m+e_v}=\mathsf{JH}W_{n,m}$. 
\item[(d)] $W_{e_w,e_w+1}\simeq\F_v[\bar \Phi_w]$.
\end{enumerate}
\end{lemma}

\begin{proof} For $i\geq -1$, we denote by
\[\Phi_{w,i}=\{\sigma\in \Phi_w\,|\,\tensor[^\sigma]\alpha{}\equiv \alpha\pmod{\pi_w^{i+1}}\mbox{ for every } \alpha\in\O_w\}\]
the lower-numbering ramification subgroups of $\Phi_w$. Using $\rho_{n,m}$ in Lemma \ref{l:basic}, 
we identify 
\[W_{1,2}=(1+\pi_w\O_w)/(1+\pi_w^2\O_w)\cong\pi_w\O_w/\pi_w^2\O_w.\]
From this, we deduce that the kernel of the action of $\Phi^0_w=\Phi_{w,0}$ on $W_{1,2}$ is given by $\Phi_{w,1}$. However, since $\Phi_{w,1}$ is a $p$-group (see, \cite[IV,\S 2, Proposition 7, Corollary 3]{ser79}) whereas $\Phi_w$ is cyclic of order dividing $p-1$, this forces $\Phi_{w,1}=\{0\}$ and proves (a).

Let  $\epsilon:\Phi_w\longrightarrow \F_p^*$ corresponds to an irreducible $\F_v[\Phi_w]$-submodule of $W_{1,2}$ and $\epsilon_0$ denote the restriction of $\epsilon$ to $\Phi_w^0$.
By viewing $\F_w$ as the permutation representation of $\Phi_w$ on $\bar \Phi_w$, we have the $\F_v[\Phi_w]$-module isomorphisms
\begin{equation}\label{e:induced}
W_{l,l+1}\simeq \pi_w^l\O_w/\pi_w^{l+1}\O_w\simeq \F_w\otimes_{\F_v} \epsilon^{\otimes l}\simeq 
\mathrm{Ind}_{\Phi_w^0}^{\Phi_w} \epsilon_0^{\otimes l},
\end{equation}
where the first isomorphism follows from $\rho_{l,l+1}$,  
the second isomorphism is induced from an isomorphism $W_{1,2}\cong \F_w\otimes_{\F_v}\epsilon$ defined as follows : first we choose (as we can) $x\in W_{1,2}$ such that $\tensor[^g]x{}=\epsilon(g)\cdot x$ for every $g\in \Phi_w$. For each $y\in W_{1,2}$, we write $y=\alpha \cdot x$ for some $\alpha \in \F_w$. Then the assignment $y\mapsto \alpha\otimes x$ is $\Phi_w$-equivariant and gives the desired isomorphism, and the third isomorphism follows from a standard result for induced representations (see, for example, \cite[p.32]{FH}). 

By (a), the order of $\epsilon_0$ is $e_w$. Therefore, by \eqref{e:decomp} and \eqref{e:induced},
$$\begin{array}{rcl}
\mathsf{JH}W_{n,n+e_w}& =& \mathsf {JH}( \bigoplus_{l=n}^{n+e_w-1}\mathrm{Ind}^{\Phi_w}_{\Phi_{w}^0}\epsilon_0^{\otimes l})
=\mathsf {JH}(\mathrm{Ind}^{\Phi_w}_{\Phi_{w}^0}(\bigoplus_{l=0}^{e_w-1}\epsilon_0^{\otimes l}))\\
 &=&\mathsf {JH}(\mathrm{Ind}^{\Phi_w}_{\Phi_{w}^0}(\F_v[\Phi_w^0]))=\mathsf {JH}\F_v[\Phi_w],
\end{array}$$
so (b) follows.
The assignment $1+\eta\mapsto 1+\pi_v\eta$, for $\eta\in \pi_w^n\O_w$, induces an isomorphism between $W_{l,l+1}$
and $W_{l+e_w,l+1+e_w}$. Hence (c) follows from \eqref{e:decomp}.  Finally, (d) is a consequence of \eqref{e:induced}.



\end{proof}


Next, we compute the $\F_p$-dimension of
$\overline{W}_{1,m}^{c_w}$, which is the same as the $\F_p$-multiplicity of $c_w$ in $\overline{W}_{1,m}$. Recall the numbers $\varphi_m$, $\psi_m$ and $\diamondsuit_m$ defined in \ref{su:special}.

\begin{lemma}\label{l:fpreg} We have  
$$\dim_{\F_p}\overline{W}_{1,m}^{c_w}=[\F_v:\F_p]\cdot (\varphi_m+\diamondsuit_m).$$
\end{lemma}
{
\begin{proof} Consider the assignment $x\mapsto x^p$ that induces a $\Phi_w$-isomorphism 
\newline
$\xymatrix{W_{1,\lceil m/p\rceil}
\ar^-\sim[r] & W_{1,m}^p}$, so $\mathsf{JH}W_{1,m}^p=\mathsf{JH}W_{1,\lceil m/p\rceil}$.
Thus, the exact sequence
$$\xymatrix{0\ar[r] & W_{\lceil m/p\rceil,m}\ar[r] & W_{1,m} \ar[r] & W_{1,\lceil m/p\rceil} \ar[r] & 0}$$
implies  $\mathsf{JH}\overline{W}_{1,m}=\mathsf{JH}W_{\lceil m/p\rceil,m}$.  

The filtration
$$W_{\lceil m/p\rceil,m}\supset W_{\lceil m/p\rceil+e_w,m}\supset \cdots \supset W_{\lceil m/p\rceil+\nu e_w,m}\cdots 
\supset W_{\lceil m/p\rceil+\varphi_m\cdot  e_w,m}  $$
gives rise to
$$\mathsf{JH}W_{\lceil m/p\rceil,m}=(\bigoplus_{\nu=0}^{\varphi_m-1} \mathsf{JH}(W_{\lceil m/p\rceil+\nu e_w,m}/ W_{\lceil m/p\rceil+(\nu+1)e_w,m}))\oplus \mathsf{JH}W_{\psi_m,m}.$$
By Lemma \ref{l:reg}, for $\nu=0,...,\varphi_m-1$, the $\Phi_w$-modules
$$\mathsf{JH}(W_{\lceil m/p\rceil+\nu e_w,m}/ W_{\lceil m/p\rceil+(\nu+1)e_w,m})=\mathsf{JH}W_{\lceil m/p\rceil+\nu e_w,\lceil m/p\rceil+(\nu+1)e_w}=\mathsf{JH}\F_v[\Phi_w],$$
so over $\F_p$, its $c_w$-eigenspace is $[\F_v:\F_p]$-dimensional. The $c_w$-eigenspace of $\mathsf{JH}W_{\psi_m,m}$ can be determined as follows.
Let $P=(a,b)$ be a non-zero point in $A_p(\bar K_v)$ as in \S\ref{su:kw}, so that $P^{(p)}=(a^p,b^p)\in E_p^{(p)}(k_w)$.
If $\mathscr{F}^{(p)}$ denotes the formal group law associated to $A^{(p)}$, then $t=-\frac{a^p}{b^p}\in
\mathfrak m_w$ corresponds to a point $Q$ in the formal group $\mathscr F^{(p)}(\mathfrak m_w)$. We have, for $\sigma\in\Phi_w$, $\tensor[^\sigma]Q{}=c_w(\sigma)\cdot Q$, while $c_w(\sigma)\in\F_p^*$, so
$$\tensor[^\sigma]t{}=c_w(\sigma)\cdot t+\text{(higher valuation terms)}.$$
Since $\ord_w(t)=\frac{n_w}{(p-1)}$, the character $c_w$ occurs in $W_{\frac{n_w}{p-1},\frac{n_w}{p-1}+1}$ with multiplicity $[\F_v:\F_p]$.
The module $\mathsf{JH}W_{\psi_m,m}$
is a direct summand of 
$\mathsf{JH}W_{\psi_m,\psi_m+  e_w}=\mathsf{JH}\F_v[\Phi_w],$ while $\F_v[\Phi_w]$ is the $\F_v$-regular representation, so its $c_w$-eigenspace has dimension $[\F_v:\F_p]\cdot \diamondsuit_m$ over $\F_p$. The result thus follows.
\end{proof}}

We can thus deduce the following result.
\begin{corollary}\label{c:mn} If $n\geq 1$, then 
$$\dim_{\F_p}\overline{W}_{n,m}^{c_w}=[\F_v:\F_p]\cdot (\varphi_m+\diamondsuit_m-\varphi_n-\diamondsuit_n).$$
\end{corollary}
\begin{proof}
As $W_{n,m}\cap W_{1,m}^p=\ker [\xymatrix{W_{1,m}^p\ar@{->>}[r] & W_{1,n}^p}]$, we obtain the exact sequence
$$\xymatrix{0 \ar[r] &\overline{W}_{n,m} \ar[r] & \overline{W}_{1,m} \ar[r] &\overline{W}_{1,n} \ar[r] & 0}$$
by the standard snake-lemma argument,
so
\begin{equation}\label{e:wnm}
\mathsf{JH}\overline{W}_{1,m}=\mathsf{JH}\overline{W}_{n,m}  \oplus \mathsf{JH}\overline{W}_{1,n}.
\end{equation}
The claimed equality thus follows by applying Lemma \ref{l:fpreg} to the above decomposition.
\end{proof}
\begin{corollary}\label{c:ebound} 
Let $\epsilon$ be as in {\em{Theorem A}}. We have
$$\log_p|\ker(\tilde{\mathfrak e}_v)/\ker(\mathfrak e_v)|=  [\F_v:\F_p]\cdot (\varphi_{\sharp}+\diamondsuit_{\sharp}-
\varphi_{\natural}-\diamondsuit_{\natural}) +\epsilon.$$
\end{corollary}
\begin{proof} 
By \eqref{e:kerevvee} and Lemma \ref{l:kertildeevvee}, 
$$\log_p|\ker(\tilde{\mathfrak e}_v)/\ker(\mathfrak e_v)|=\log_p
|\overline W_{1,\sharp}^{c_w}|-\log_p|\overline W_{1,\natural_w}^{c_w}|+\epsilon.$$
Then the result follows from Corollary \ref{c:ebound}.
\end{proof}

\subsection{Artin-Schreier extensions}\label{su:artsch} 

In this subsection, we will make a substantial use of the theory of Artin-Schreier extensions. The main reference is \cite[\S 2]{thd20}, especially its Lemma 2.1, 2.2, and 2.3. 

The extension $k'_{w'}/k_w$ is given by an Artin-Schreier equation $Y^p-Y=\kappa$ for some $ \kappa\in k_w$. Two elements $ \kappa, \kappa'\in k_w$ yield the same extension if and only if 
\begin{equation}\label{e:omega'}
 \kappa'=\alpha^p-\alpha+a\cdot  \kappa,\;\;\text{for some}\;\alpha\in k_w, \;a\in \F^*_w.
\end{equation}

It suffices to assume that $k'_{w'}/k_w$ is ramified so that $\ord_w \kappa<0$. To see this, suppose $\O_w=\F_{w}[[\pi_w]]$ and write $ \kappa=\sum_{i} a_i\pi_w^i$, $a_i\in\F_w$.  If $\ord_w\kappa>0$, then we can find $\alpha=\sum_{j=1}^\infty b_j\pi_w^j$, $b_j\in\F_w$ such that, $\kappa=\alpha-\alpha^p$ but this contradicts with the fact that $k'_{w'}/k_w$ is a 
non-trivial field extension.  On the other hand, if $\ord_w\kappa=0$, then there is some $\alpha\in K_w$ such that $\kappa-(\alpha^p-\alpha)=a_0\in\F_w$, and hence $k_{w'}'/k_w$ is unramified. 

Now within the collection
$$ \{ \kappa' \; \mid k_{w'}'=k_w(z),\; z^p-z=\kappa'\in k_w\},$$ 
we pick an element $\kappa$ whose $w$-adic valuation is maximal.
Then by \cite[Lemma 2.3]{thd20}, 
$$\lambda_w:=-\ord_w\kappa,$$
where $\lambda_w=f_w-1$ as defined in \S\ref{su:special}. It is useful to note that we have
\begin{equation}\label{e:fwfv}
\lambda_w=e_w\lambda_v.
\end{equation}
Recall that $e_w$ denotes the ramification index of $k_w/K_v$.
Since $k_w/K_v$ is unramified or tamely ramified, $\ord_w\mathrm{Disc}(k_w/K_v)=[\F_w:\F_v](e_w-1)$.
As, in \S\ref{su:chcond}, the equality \eqref{e:fwfv} 
follows from the conductor-discriminant formula and the transitivity of 
the discriminant.

We claim that
\begin{equation}\label{e:plw}
p\nmid \lambda_w.
\end{equation}
Otherwise we can write
$\lambda_w=p\nu$, and 
$$\kappa=a_{-\lambda_w}\pi_w^{-\lambda_w}+\text{(higher valuation terms)}= \alpha^p+\text{(higher valuation terms)},$$
with $\alpha=b\pi_w^{-\nu}$, $b^p=a_{-\lambda_w}$. Hence the element $\kappa':=\kappa-(\alpha^p-\alpha)$
would have higher $w$-adic valuation which contradicts our choice of $\kappa$.

\subsubsection{The associated $\kappa_h$}\label{ss:kappah}
For the chosen $\kappa$, let $y\in k'_{w'}$ satisfying
$y^p-y=\kappa$, so that
\begin{equation}\label{e:y}
\ord_w(y)=-\frac{1}{p}\lambda_w>-\lambda_w=\ord_w(y^p).
\end{equation}

Write 
$$-\lambda_w =-p\vartheta_w+\rho_w,$$
where $\vartheta_w,\rho_w\in \Z$, $\vartheta_w> 0$, 
$1\leq \rho_w\leq p-1$. 

For each integer $h$, choose an integer $\nu_h$ such that 
\[\nu_h\cdot \rho_w\equiv h\pmod{p}.\]
Put $r_h:=\frac{h-\nu_h\cdot \rho_w}{p}$ and
\[\kappa_h:=\kappa^{\nu_h}\cdot \pi_w^{p(\nu_h\vartheta_w+r_h)}\]
which is an element in $k_w$ and is dependent on the choice of $\nu_h$. 
\begin{lemma}\label{l:R_h} For every integer $h$, one can write in $\O_{w'}$
$$\kappa_h=\mathfrak{A}_h+\mathfrak{R}_h$$ where $\mathfrak A_h\in \O_{w'}^p$ and
$\ord_w(\mathfrak{A}_h)=h$, $\ord_w(\mathfrak{R}_h)=h+\frac{(p-1)\lambda_w}{p}$. 
\end{lemma}
\begin{proof}
First we note that by construction, we have $\ord_w (\kappa_h)=h$. Moreover, 
\[\kappa\cdot \pi_w^{p\vartheta_w}=(y^p-y)\cdot  \pi_w^{p\vartheta_w}= (y\cdot \pi_w^{\vartheta_w})^p -(y\cdot \pi_w^{p\vartheta_w}).\]
Note by taking $w$-adic valuation that
\begin{equation}\label{valuation}
\ord_w(y\cdot \pi_w^{\vartheta_w})^p=(p-1)\ord_w(y)+\ord_w(y\cdot \pi_w^{p\vartheta_w})=-\dfrac{(p-1)\lambda_w}{p}+\ord_w(y\cdot \pi_w^{p\vartheta_w}), 
\end{equation}
where by \eqref{e:y}, the term $(y\cdot\pi_w^{\vartheta_w } )^p$ has strictly lower $w$-valuation than $y\cdot \pi_w^{p\vartheta_w  }$.

Therefore, the standard binomial expression allows us to write
\[(\kappa\cdot \pi_w^{p\vartheta_w})^{\nu_h}=\left((y\cdot \pi_w^{\vartheta_w})^p -(y\cdot \pi_w^{p\vartheta_w})\right)^{\nu_h}=(y\cdot\pi_w^{\vartheta_w})^{p\nu_h}+R_h \mbox{ with } R_h\in \O_{w'}.\]
Here we have
\begin{equation}\label{val2}\ord_w(R_h)=(\nu_h-1)\ord_w(y\cdot \pi_w^{\vartheta_w})^p+\ord_w(y\cdot \pi_w^{p\vartheta_w})\stackrel{(\ref{valuation})}{=}\ord_w(y\cdot \pi_w^{\vartheta_w})^{p\nu_h}+\dfrac{(p-1)\lambda_w}{p}.
\end{equation}
Consequently by setting $\mathfrak{A}_h=(y^{\nu_h}\cdot \pi_w^{\theta_w\nu_h+r_h})^p$ and $\mathfrak{R}_h=R_h\cdot \pi_w^{pr_h}$, one can write 
\begin{equation}\label{val3}
    \kappa_h=(\kappa\cdot \pi_w^{p\vartheta_w})^{\nu_h}\cdot \pi_w^{pr_h}=\mathfrak{A}_h+\mathfrak{R}_h
\end{equation}
for which, by (\ref{val2}), $\ord_w(\mathfrak{R}_h)=\ord_w(\mathfrak{A}_h)+\frac{(p-1)\lambda_w}{p}$. Since by construction we have $\ord_w(\kappa_h)=h$, we conclude that $\ord_w(\mathfrak{A}_h)=h$ and hence $\ord_w(\mathfrak{R}_h)=h+\frac{(p-1)\lambda_w}{p}$.\end{proof}

\subsubsection{The proof of Theorem A}\label{ss:pfa}
We shall show that Theorem A is a consequence of putting \eqref{e:firstin}, Corollary \ref{c:ebound} together as well as Lemma \ref{l:thirdin}, Lemma \ref{l:small}  below.

\begin{lemma}\label{l:cpmain} Let $x=1+\mathfrak X$ be an element of $W_{w,1}$. To have  
$x\in W_{w',1}^p\cdot W_{w',\natural_{w'}}$ hold in $W_{w',1}$, it is necessary and sufficient that $\ord_w (\mathfrak X)\geq \natural_w-\frac{(p-1)\lambda_w}{p}$.

\end{lemma}
\begin{proof}  Because $k'_{w'}/k_w$ is totally ramified, $\ord_{w'}(z)=p\cdot \ord_w(z)$.
Thus, $\natural_{w'}=p\cdot\natural_w$ and in $W_{{w'},1}$, we already have
$ W_{w,\natural_w}\subset W_{w',\natural_{w'}}$. So, we may assume that $n:=\ord_w (\mathfrak X)<\natural_w$.

The  $\F_w$-vector space $W_{w,h,h+1}$ is $1$-dimensional, its elements can be represented by elements of the form $1+{\alpha\kappa_h}$ for some $\alpha\in\F_w$, hence each $x_h\in W_{w,h}$ can be written as
$x_h=(1+\alpha_h\kappa_h)\cdot y_h$, $y_h\in W_{w,h+1}$. The argument applied inductively yields 
\begin{equation}\label{e:xym}
x=\prod_{h=n}^{\natural_w-1} (1+\alpha_h\kappa_h)\cdot y,\;\text{where}\;y\in W_{w,\natural}.
\end{equation}
Put $a=\alpha_h\kappa_h$
and $b=a+\alpha_h^2\mathfrak A_h\mathfrak R_h$, where $\mathfrak A_h$ and $\mathfrak R_h$ are as in Lemma \ref{l:R_h}. By \eqref{e:basic},
\begin{equation}\label{e:bb}
1+\alpha_h\kappa_h= (1+b)\cdot (1+c)=(1+\alpha\mathfrak A_h)(1+\alpha\mathfrak R_h)\cdot (1+c)=(1+\mathfrak B_h)^p(1+\mathfrak D_h).
\end{equation}
This holds in $W_{w',1}$, here $\mathfrak B_h^p=\alpha_h\kappa_h$ and 
$\ord_{w}(\mathfrak D_h)=h+\frac{(p-1)\lambda_w}{p}$, because 
$$\ord_w(c)=\ord_w(a-b)=2h+\frac{(p-1)\lambda_w}{p}>h+\frac{(p-1)\lambda_w}{p}=\ord_w(\mathfrak R_h).$$
 Therefore, if $n \geq \natural_w-\frac{(p-1)\lambda_w}{p}$, then by \eqref{e:xym} and \eqref{e:bb}, 
$x\in W_{w',1}^p\cdot W_{w',\natural_{w'}}$ as desired.

In general, by \eqref{e:xym} and \eqref{e:bb}, we can write $x=(1+\mathfrak B)^p\cdot (1+\mathfrak D)$, where
$$(1+\mathfrak B)=\prod_{h=n}^{\natural_w-1} (1+\mathfrak B_h),\;
(1+\mathfrak D)=y\cdot \prod_{h=n}^{\natural_w-1}(1+\mathfrak D_h).$$
Suppose $x\in W_{w',1}^p\cdot W_{w',\natural_{w'}}$ so that $x=(1+\mathfrak E)^p\cdot (1+\mathfrak C)^{-1}$.
Then
$$(1+\mathfrak F):=(1+\mathfrak D)(1+\mathfrak C)=(1+\mathfrak B)^{-p}(1+\mathfrak E)^p.$$
This implies that $\mathfrak F\in \O_{w'}^p$. It follows that $n+\frac{(p-1)\lambda_w}{p}\geq \natural_w$,
for otherwise,
$$\ord_{w'} \mathfrak F=\ord_{w'} \mathfrak D=pn+(p-1)\lambda_w$$
that, by \eqref{e:plw}, is not divisible by $p$.
\end{proof}
\begin{remark}Lemma \ref{l:cpmain} can be interpreted as an equality
\begin{equation}\label{e:cpmain}
W_{w,1}\cap (W_{w',1}^p\cdot W_{w',\natural_{w'}})=W_{w, \flat_w}.
\end{equation}    
\end{remark}


\begin{lemma}\label{l:thirdin} We have
\begin{equation}\label{e:thirdin}
\log_p|\ker(\tilde{\mathfrak c}_v)/\ker(\mathfrak c_v)|= [\F_v:\F_p]\cdot \underline{\mathrm b}_v.
\end{equation}
We have $n_v\geq \underline{\mathrm b}_v\geq 0$. If $\flat_w=1$, then $\underline{\mathrm b}_v=n_v$.
Also, $\underline{\mathrm b}_v>0$, if $\lambda_v>1$.
\end{lemma}

\begin{proof} It follows from \eqref{e:kercvcv} and \eqref{e:cpmain} that
$\ker(\tilde{\mathfrak c}_v)/\ker(\mathfrak c_v)=\overline{W}_{w,\flat_w,\natural_w}^{c_w}$ which has dimension
$\underline{\mathrm b}_v$ over $\F_v$. This proves \eqref{e:thirdin}. Then observe that since $\overline{W}_{w,\flat_w,\natural_w}^{c_w}\subset \overline{W}_{w,1,\natural_w}^{c_w}$,
$$[\F_v:\F_p]\cdot \underline{\mathrm b}_v\leq \log_p|\overline{W}_{w,1,\natural_w}^{c_w}|=[\F_v:\F_p]\cdot
(\varphi_{\natural_w}-\diamondsuit_{\natural_w})=[\F_v:\F_p]\cdot
(n_v-0).$$
Finally, if $\lambda_v\geq \frac{p}{p-1}$, then by \eqref{e:fwfv},
$\lambda_w\geq \frac{pe_w}{p-1}$ and $\natural_w\geq \flat_w+e_w$, hence
$\mathrm{JH}\overline W_{w,\flat_w,\natural_w}$ contains the submodule
$\mathrm{JH}\overline W_{w,\natural_w-e_w,\natural_w}$, which by
Lemma \ref{l:reg}, equals $\mathrm{JH} \F_v[\Phi_w]$.
This shows that $\log_p|\overline W_{w,\flat_w,\natural_w}^{c_w}|\geq [\F_v:\F_p]$.
\end{proof}

We end the proof of Theorem A by proving the lemma below.

\begin{lemma}\label{l:small}
If $\lambda_v=1$, then 
$\overline{\mathrm b}_v=\underline{\mathrm b}_v=0$, hence $\coh^1(G_v, A(K'_{v'}))$ is trivial. 
\end{lemma}
\begin{proof} It follows that $\lambda_w=e_w\leq p-1$. We need to show 
$$\varphi_{\sharp_w}+\diamondsuit_{\sharp_w}=\varphi_{\flat_w}+\diamondsuit_{\flat_w}.$$

These are proven by direct calculation. From the definition, $\natural_w=\frac{p}{p-1}n_w=\frac{pe_wn_v}{p-1}$,
so $\sharp_w=\lceil \natural_w+\frac{(p-1)\lambda_w}{p}\rceil=\natural_w+\lambda_w+\lceil \frac{-\lambda_w}{p}\rceil=\natural_w+\lambda_w$, $\flat_w=\lceil \natural_w-\frac{(p-1)\lambda_w}{p}\rceil=\natural_w-\lambda_w+1$.

Then $\lceil \frac{\sharp_w}{p}\rceil=\frac{n_w}{p-1}+\lceil \frac{\lambda_w}{p}\rceil=\frac{n_w}{p-1}+1$, and
$\varphi_{\sharp_w}=\lfloor \frac{\frac{pn_w}{p-1}+\lambda_w-\frac{n_w}{p-1}-1}{e_w} \rfloor
=n_v+1+\lfloor \frac{-1}{e_w}\rfloor=n_v$. Consequently, 
$\psi_{\sharp_w}=\frac{n_w}{p-1}+1+n_v\cdot e_w=\natural_w+1$. Since 
$\sharp_w=\frac{n_w}{p-1}+n_w+e_w\equiv \frac{n_w}{p-1}\pmod{e_w}$, and the interval $[\psi_{\sharp_w},\sharp_w)$ is of length $e_w-1$,
we have $\diamondsuit_{\sharp_w}=0$. So, $\varphi_{\sharp_w}+\diamondsuit_{\sharp_w}=n_v$.

If $p=2$, then $\flat_w=2n_v$,
so $\varphi_{\flat_w}=\lfloor \frac{2n_v-\lceil 2n_v/2\rceil}{p-1} \rfloor=n_v$.
Thus, $\psi_{\flat_w}=n_v+n_v\cdot 1=2n_v$. Since $[\psi_{\flat_w},\flat_w)=\emptyset$, $\diamondsuit_{\flat_w}=0$. So, $\varphi_{\flat_w}+\diamondsuit_{\flat_w}=n_v$.
If $p>2$, then 
$\lceil \frac{\flat_w}{p}\rceil=\frac{n_w}{p-1}+\lceil \frac{-\lambda_w+1}{p}\rceil =\frac{n_w}{p-1}$,
so $\varphi_{\flat_w}=\lfloor \frac{\frac{pn_w}{p-1}-e_w+1-\frac{n_w}{p-1}}{e_w} \rfloor=n_v-1$, 
$\psi_{\flat_w}=\frac{n_w}{p-1}+(n_v-1)\cdot e_w$. 
Since $\psi_{\flat_w}\equiv \frac{n_w}{p-1}\pmod{e_w}$, we have $\diamondsuit_{\flat_w}=1$.
Therefore,  $\varphi_{\flat_w}+\diamondsuit_{\flat_w}=n_v$, as well.

\end{proof}
This completes the proof of Theorem A.

\section{Local cohomology}\label{s:locc}
In this section, we establish several preliminary results concerning the size of local cohomology groups,
and apply them to the proof of Theorem B.
At the outset, let us assume that $A$ is an ordinary abelian variety defined over $\mathsf K$ which is a
completion of a global field. Let
$\mathsf K'/\mathsf K$ be a finite Galois extension, $\mathsf G:=\Gal(\mathsf K'/\mathsf K)$.

\subsection{Finite Galois extensions}\label{su:locc} 
 We recall the following, which could be already known to
experts. Here $\mathsf G$ is not necessary abelian.

\begin{lemma}\label{l:F'w'} The cohomology group $\coh^1(\mathsf G, A(\mathsf K'))$ is finite.
\end{lemma}
\begin{proof} Since $\coh^1(\mathsf G,  A(\mathsf K'))$ is the Pontryagin dual of 
$$Q:=A^t(\mathsf K)/\Nm_{\mathsf G}(A^t(\mathsf K'))$$
\cite[Corollary 2.3.3]{tan10},
it is sufficient to show that $Q$ 
is finite for every $A$. 

If $\mathsf K$ is Archimedean,
then $A/\mathsf K$ is a compact Lie group and the multiplication by  $[\mathsf K':\mathsf K]$ on $A(\mathsf K')$ is an open map,
so $A(\mathsf K)/[\mathsf K':\mathsf K]\cdot A(\mathsf K)$ is a discrete compact group, hence finite,
consequently, its quotient $Q$ is also finite. 

In the non-Archimedean case, let $\mathsf O$ and $\mathsf m$ (resp. $\mathsf O'$ and $\mathsf m'$)
be the ring of integers and maximal ideal associated to $\mathsf K$ (resp. $\mathsf K'$).
Since the special fibre of $A/\mathsf K$ has finite number of rational points over the residue field of $\mathsf K$, we only need to consider the group $A_o(\mathsf K)$ of $\mathsf K$-rational points whose reduction is the identity. Write $g:=\dim A$. The group $A_o(\mathsf K)$ can be parametrized as $\mathscr F(\mathsf m)$ where $\mathscr F$ is the associated formal group law 
given by a family $f(X,Y)=(f_i(X,Y))$ of $g$ formal power series in $2g$ variables $X_j,Y_j$, $1\leq i,j\leq g$, with coefficients in $\mathsf O$, so that on the formal group $\mathscr F(\mathsf m)$ the addition is given by
$$x\oplus y=f(x,y).$$
Consider the trace map $\N_{\mathsf G}:\mathscr F(\mathsf m')\longrightarrow \mathscr F(\mathsf m)$ such that
$$\N_{\mathsf G}(z)=\tensor[^{\tau_1}]z{}\oplus \cdots\oplus \tensor[^{\tau_e}]z{},$$
where $\{\tau_1,...,\tau_e\}=\mathsf G$. Since 
$$f(X,Y)=X+Y+\text{higher degree terms},$$
we have (see \cite[Lemma 3.4]{cog96} and its proof) that, for $z\in \prod_{i=1}^g (\mathsf m')^s$,
\begin{equation}\label{e:cg}
\N_{\mathsf G}(z)=\text{Tr}_{\mathsf G}(z)+y,\;\;y\in \prod_{i=1}^g (\mathsf m')^{2s}.
\end{equation}

Let $\mathsf e$ and $\mathsf d$ denote the ramification index and the different of $\mathsf K'/\mathsf K$.
By \cite[VIII, Proposition 4]{we74},  
$\mathsf m^\alpha\subset \text{Tr}_{\mathsf G}((\mathsf m')^{\mathsf e\alpha-\mathsf d})$.
Let $t_0$ be an integer greater than $2\mathsf d$.
Then $2(\mathsf e t_0-\mathsf d)>t_0$. If $x^{(0)}\in \prod_{i=1}^g\mathsf m^{t_0}$, then 
$x^{(0)}=\text{Tr}_{\mathsf G}(z^{(1)})$ for some $z^{(1)}\in \prod_{i=1}^g (\mathsf m')^{\mathsf e t_0-\mathsf d}$. By \eqref{e:cg}, there is some $x^{(1)}\in \prod_{i=1}^g\mathsf m^{t_1}$, $t_1>t_0$ such that
$$x^{(0)}=\N_{\mathsf G}(z^{(1)})\oplus x^{(1)}.$$
By repeating this process, we obtain a strictly increasing sequences of natural numbers
$t_n$ ($n=0,1,...$) together with sequences $x^{(n)}\in \prod_{i=1}^g \mathsf m^{t_n}$
and $z^{(n+1)}\in \prod_{i=1}^g (\mathsf m')^{\mathsf e t_n+\mathsf d}$ that
satisfy
$$x^{(n)}=\N_{\mathsf G}(z^{(n+1)})\oplus x^{(n+1)}.$$

Write 
$z_{[n]}$ for $z^{(1)}\oplus z^{(2)}\oplus\cdots\oplus z^{(n)}$. Then $\mathcal N_{\mathsf G}(z_{[n]})\equiv x^{(0)}\pmod{\prod_{i=1}^g \mathsf m^{t_n}}$, and for $m>n$, we have the congruence 
$$z_{[m]}\equiv z_{[n]} \pmod{\prod_{i=1}^g (\mathsf m')^{\mathsf e t_{n+1}+\mathsf d}}.$$
Therefore,  $z_{[n]}$ converges to an element $z\in \prod_{i=1}^g (\mathsf m')^{\mathsf e t_0+\mathsf d}$ such that $\N_{\mathsf G}(z)=x^{(0)}$. This shows $\mathscr F(\mathsf m^{t_0})\subset \N_{\mathsf G}(\mathscr F(\mathsf m'))$.  Consequently,
$\mathscr F(\mathsf m)/\N_{\mathsf G}(\mathscr F(\mathsf m'))$ is a quotient of the finite group
$\mathscr F(\mathsf m)/\mathscr F(\mathsf m^{t_0})$ and hence is also a finite group.

Thus, the lemma is proved as long as the N\'{e}ron-model of $A$ remains stable under $\mathsf K'/\mathsf K$, because then 
$\mathscr F\times_{\mathsf O}\mathsf O' $ is the formal group law associated to $A/\mathsf K'$ and the group of $\mathsf O'$-points on the identity component of $A/\mathsf K'$ can be identified with $\mathscr F(\mathsf m')$. 

In general, if $\mathscr F'$ denote the formal group law associated to $A/\mathsf K'$ with the addition $\oplus'$ and 
$\N'_{\mathsf G}:\mathscr F'(\mathsf m')\longrightarrow \mathscr F(\mathsf m)$ denotes the trace map defined by
$$\N'_{\mathsf G}(z)=\tensor[^{\tau_1}]z{}\oplus' \cdots\oplus' \tensor[^{\tau_e}]z{},$$
the N\'{e}ron mapping property implies there is a homomorphism $\varrho:\mathscr F\times_{\mathsf O}\mathsf O' \longrightarrow \mathscr F'$ that fits into the commutative diagram
$$\xymatrix{\mathscr F(\mathsf m') \ar[r]^-\varrho \ar[d]_-{\N_{\mathsf G}} & \mathscr F'(\mathsf m') \ar[ld]^-{\N'_{\mathsf G}}\\
\mathscr F(\mathsf m) &}.$$
This shows that $\mathscr F(\mathsf m)/{\N'_{\mathsf G}}(\mathscr F'(\mathsf m'))$
is a quotient of $\mathscr F(\mathsf m)/{\N_{\mathsf G}}(\mathscr F(\mathsf m'))$ and hence is a finite group.
\end{proof}
{\begin{remark}
Note that the proof shows that if $\mathsf K'/\mathsf K$ is unramified, then $\mathscr F(\mathsf m)={\N_{\mathsf G}}(\mathscr F(\mathsf m'))$.
\end{remark}}
Let $\Pi$ denote the group of connected components (over $\bar\F_{\mathsf K}$)
of the special fibre of the N\'{e}ron model of $A$ over $\O_{\mathsf K}$. If $\mathsf K^{un}$ denotes the maximal unramified extension of $\mathsf K$, then by \cite[\S I.3.8]{mil86}, $\coh^1(\mathsf K^{un}/\mathsf K, A(\mathsf K^{un}))=\coh^1(\bar\F_{\mathsf K}/\F_{\mathsf K},\Pi)$. Here the isomorphism is induced from the reduction map.

\begin{lemma}\label{l:unrambound} 
If $\mathsf K'/\mathsf K$ is unramified, then $|\coh^1(\mathsf G, A(\mathsf K'))|\leq |\Pi_{\mathsf K}|$.
In particular, if $A$ has good reduction, then $\coh^1(\mathsf G, A(\mathsf K'))=0$.
\end{lemma}
{
\begin{proof}
If $\mathsf K'/\mathsf K$ is unramified, then the reduction map described above induces an injection
\begin{equation}\label{e:milne}
\xymatrix{\coh^1(\mathsf G, A(\mathsf K')) \ar@{^(->}[r] & \coh^1(\F_{\mathsf K'}/\F_{\mathsf K},\Pi_{\mathsf K'})},
\end{equation}
where $\Pi_{\mathsf K'}:=\Pi^{\Gal(\bar\F_{\mathsf K}/\F_{\mathsf K'})}$. Since $\Gal(\F_{\mathsf K'}/\F_{\mathsf K})=\mathsf G$ is cyclic, we have
$$|\coh^1(\mathsf G, A(\mathsf K'))|\leq |\coh^1(\F_{\mathsf K'}/\F_{\mathsf K},\Pi_{\mathsf K'})|=|\coh^2(\F_{\mathsf K'}/\F_{\mathsf K},\Pi_{\mathsf K'})|=|\Pi_{\mathsf K}/\Nm_G(\Pi_{\mathsf K'})|\leq |\Pi_{\mathsf K}|.
$$
\end{proof}

}
\begin{corollary}\label{c:unrambound}
Let $\mathsf K_\infty/\mathsf K$ be the unramified $\Z_p$-extension. Then, for $i=1,2$,
$$|\coh^i(\mathsf K_\infty/\mathsf K,A(\mathsf K_\infty))|\leq |\Pi_{\mathsf K}|.$$
\end{corollary}
\begin{proof}Since $\coh^i(\mathsf K_\infty/\mathsf K,A(\mathsf K_\infty))$ is the injective limit of 
$\coh^i(\mathsf K_n/\mathsf K,A(\mathsf K_n))$, where $\mathsf K_n$ denotes the $n$-th layer,
it is sufficient to show $|\coh^i(\mathsf K_n/\mathsf K,A(\mathsf K_n))|\leq |\Pi_{\mathsf K}|$.
For $i=1$, this is directly from the lemma. By the remark following the proof of Lemma \ref{l:F'w'},
$$\coh^2(\mathsf K_n/\mathsf K,A(\mathsf K_n))=A(\mathsf K)/\Nm_{\mathsf K_n/\mathsf K}(A(\mathsf K_n))=\coh^2(\F_{\mathsf K'}/\F_{\mathsf K},\Pi_{\mathsf K'})=\Pi_{\mathsf K}/\Nm_G(\Pi_{\mathsf K'}).$$

\end{proof}

\subsection{Semi-stable reductions}\label{su:ssred}
Suppose $A/\mathsf K$ has semi-stable reduction. Then there is an unramified extension 
 $\mathsf L$ of $\mathsf K$ such that the reduction of $A$
over $\mathsf L$ is an extension of an abelian variety by a {\em split} torus of dimension $\mathsf d$. Over $\mathsf L$, $A$ admits a split uniformization in the sense that 
there is semi-abelian $\mathsf L$-group $\mathsf E$ given by the exact sequence 
\begin{equation}\label{e:teb}
\xymatrix{0\ar[r] & \mathsf  T \ar[r]  & \mathsf  E \ar[r] & \mathsf  B\ar[r] &0}
\end{equation}
where $\mathsf  B$ is an abelian variety and $\mathsf  T=\mathbb G_m^{\mathsf g}$ such that in rigid geometry, $A$ equals the quotient of $E$ by a discrete lattice $\mathsf\Omega$ of rank $\mathsf g\leq \dim A$ (see \cite[\S 1]{bl91} and \cite[\S 1]{bx96}).  Furthermore, if $\mathsf F$ contains $\mathsf L$, then the component group $\Pi_{\mathsf F}$
of $A/\mathsf F$ is a finite quotient of $\Z^{\mathsf g}$ (see \cite[Proposition 5.2]{bx96}).


\subsubsection{Unramified extensions}\label{ss:unramf} The lemma below provides an upper bound of $|\coh^1(\mathsf G, A(\mathsf K'))|$ different from that in Lemma \ref{l:unrambound}. We say that an abelian group is spanned by at most
$\mathsf f$ elements if it is a quotient of $\Z^{\mathsf f}$.
In particular, the aforementioned $\Pi_{\mathsf F}$ is spanned by at most $\dim A$ elements.
If an abelian group $N$ is spanned by at most $\mathsf f$ elements, then so is any of its 
sub-quotients; if in addition, $N$ is annihilated by an integer $n$, then $N$ is a quotient of $(\Z/n\Z)^{\mathsf f}$, and hence
$\log_p|N|\leq \mathsf f\cdot  \log_p n$.

\begin{lemma}\label{l:splitbd}
If $A/\mathsf K$ has semi-stable reduction and $\mathsf K'/\mathsf K$ unramified, then 
$\coh^1(\mathsf G, A(\mathsf K'))$ is spanned by at most
$\dim A$ elements and hence has order bounded by $|\mathsf G|^{\dim A}$.
\end{lemma}
\begin{proof} Write $\mathsf L'=\mathsf K'\mathsf L$, $\mathsf G':=\Gal(\mathsf L'/\mathsf K)$, where $\mathsf L'/\mathsf K$ is unramified and $\mathsf G'$ is cyclic. Since $|\mathsf G|$ annihilates $\coh^1(\mathsf G, A(\mathsf K'))$ and by the inflation map
$$\xymatrix{\coh^1(\mathsf G, A(\mathsf K'))\ar@{^(->}[r] & 
\coh^1(\mathsf G', A(\mathsf L'))\subset \coh^1(\mathsf G',\Pi_{\mathsf L'}),}$$
it is sufficient to show that $\coh^1(\mathsf G',\Pi_{\mathsf L'})$ is spanned by at most
$\dim A$ elements.

Let $\mathsf g\in \mathsf G'$ be a generator, then each inhomogeneous $1$-cocycle $\rho$ is determined by $\rho(\mathsf g)$. 
Since $\Pi_{\mathsf L'}$ is spanned by at most $\dim A$ elements, the group of $\Pi_{\mathsf L'}$- valued inhomogeneous $1$-cocycles is spanned by at most $\dim A$ elements, and so is its sub-quotient 
$\coh^1(\mathsf G', A(\mathsf L'))$.
\end{proof}

\subsubsection{Ordinary reduction}\label{ss:ordinary}

\begin{lemma}\label{l:fin} Suppose $A/\mathsf K$ has ordinary reduction. If $\mathsf M/\mathsf K$ is a finite abelian $p$-extension whose Galois group $\mathsf \Psi$ is spanned by at most $\mathsf d$ elements, then 
the group
$\coh^1(\mathsf\Psi,A(\mathsf M))$ is spanned by at most $(\mathsf d+1)\dim A$ elements, 
and hence for every positive integer $\nu$, 
\begin{equation*}
\log_p|\coh^1(\mathsf\Psi, A(\mathsf M))_{p^\nu}|\leq \nu\cdot {(\mathsf d+1)\dim A}.
\end{equation*}
\end{lemma}
\begin{proof}If the characteristic of $\mathsf K$ is $0$, this follows from a theorem of Mazur (see \cite[Corollary 4.30]{maz}). 
In general, the lemma can be deduced from \cite{tan10}. We reproduce the argument as follows.

Suppose $A$ has good ordinary reduction. Because $\O_{\mathsf M}$ is Henselian and the N$\acute{\text{e}}$ron model of $A$ over $\O_{\mathsf M}$ is smooth, the reduction map $A(\mathsf M)\longrightarrow \bar A(\F_{\mathsf M})$ is surjective (see \cite[\S I.4.13]{mil80}). Let $\hat A$ denote the associated formal group.
We have the induced exact sequence
\begin{equation}\label{e:hatbar}
\xymatrix{0\ar[r] & \coh^1(\mathsf\Psi, \hat A(\O_{\mathsf M} )) \ar[r] & \coh^1(\mathsf\Psi, A(\mathsf M )) \ar[r] &\coh^1(\mathsf\Psi,\bar A(\F_{\mathsf M}))}.
\end{equation}
According to \cite[Corollary 2.3.3]{tan10}, the local duality of Tate gives rise to the 
identification 
$$\coh^1(\mathsf\Psi, A(\mathsf M ))=(A^t(\mathsf K )/\Nm_{\mathsf\Psi}(A^t(\mathsf M )))^\vee,$$ 
also, by [Ibid, Proposition 2.6.1], a $p$-primary element of $\coh^1(\mathsf K,A)$ is in
$\coh^1(\mathsf K, \hat{A}):=\coh^1(\Gal(\bar{\mathsf K}^s/\mathsf K), \hat{A}(\O_{\bar{\mathsf K}^s}))$
if and only if it annihilates  $\hat{A}^t(\O_{\mathsf K})$. Therefore, in $\coh^1(\mathsf K, A)$, the discrete subgroup  
$$\coh^1(\mathsf\Psi, \hat A(\O_{\mathsf M} ))=\coh^1(\mathsf K, \hat{A})\cap \coh^1(\mathsf\Psi, A(\mathsf M ))$$
is the annihilator of $\hat{A}^t(\O_{\mathsf K})\cdot \Nm_{\mathsf\Psi}(A^t(\mathsf M ))$, in other words,
$$\coh^1(\mathsf\Psi, \hat A(\O_{\mathsf M} ))=(\hat{A}^t(\O_{\mathsf K})\backslash A^t(\mathsf K )/\Nm_{\mathsf\Psi}(A^t(\mathsf M )))^\vee.$$ 
Since $\hat{A}^t(\O_{\mathsf K})\backslash A^t(\mathsf K )
=\bar A^t(\F_{\mathsf K})$, we see that $\coh^1(\mathsf\Psi, \hat A(\O_{\mathsf M} ))=\bar A^t(\F_{\mathsf K})/\Nm_{\mathsf\Psi}(\bar A^t(\F_{\mathsf M}))$,
as well as the $p$-primary part of $\bar A^t(\F_{\mathsf K})$ is spanned by at most $\dim A$ elements.

The group of inhomogeneous $\bar A(\F_{\mathsf M})$-valued $1$-cocycles of $\mathsf\Psi$
is spanned by at most $\mathsf d\cdot\dim A$ elements, so is its quotient
$\coh^1(\mathsf\Psi,\bar A(\F_{\mathsf M}))$. Thus, by \eqref{e:hatbar}, the group
$\coh^1(\mathsf\Psi, A(\mathsf M ))$ is spanned by at most $(\mathsf d+1)\dim A$ elements.

Suppose $A$ has split-multiplicative reduction. The computation in the last paragraph of \cite{tan10} asserts that $\coh^1(\mathsf\Psi, A(\mathsf M))$ is isomorphic to a subgroup of $\Hom(\mathsf\Psi, (\Q/\Z)^{\dim A})$, and hence 
$\coh^1(\mathsf\Psi, A(\mathsf M))$ is spanned by at most $\mathsf d\cdot \dim A$ elements.

Suppose $A/\mathsf K$ has multiplicative reduction and let $\mathsf L$ be as above so that in the
restriction-inflation exact sequence
\begin{equation}\label{e:klma}
\xymatrix{\coh^1(\mathsf L/\mathsf K,A(\mathsf L)) \ar[r] & \coh^1(\mathsf M\mathsf L/\mathsf K, A(\mathsf M\mathsf L)) \ar[r] &
\coh^1(\mathsf M\mathsf L/\mathsf L, A(\mathsf M\mathsf L))}
\end{equation}
the right item $\coh^1(\mathsf M\mathsf L/\mathsf L, A(\mathsf M\mathsf L))$ is known to be spanned by
at most
$\mathsf d\cdot\dim A$ elements (since $A/\mathsf L$ has split-multiplicative reduction). 
By Lemma \ref{l:splitbd}, the left item is spanned by at most $\dim A$ elements (since $L/K$ is unramified), so $\coh^1(\mathsf M\mathsf L/\mathsf K, A(\mathsf M\mathsf L))$
is spanned by at most $(\mathsf d+1)\dim A$ elements. 
Then the (injective) inflation map 
$$\xymatrix{\coh^1(\mathsf M/\mathsf K, A(\mathsf M))\ar@{^(->}[r] & \coh^1(\mathsf M\mathsf L/\mathsf K, A(\mathsf M\mathsf L))}$$
{implies that the first group is also spanned by at most $(\mathsf d+1)\dim A$ elements.}
\end{proof}
\begin{corollary}\label{c:fin} If $A/\mathsf K$ has ordinary reduction and $\mathsf M/\mathsf K$ is a $\Z_p^{\mathsf d}$-extension, then 
\begin{equation}\label{e:unibd}
\log_p |\coh^1(\mathsf M/\mathsf K, A(\mathsf M))_{p^\nu}|\leq \nu\cdot {(\mathsf d+1)\dim A}.
\end{equation}
In particular, $\coh^1(\mathsf M/\mathsf K,A(\mathsf M))^\vee$ is topologically spanned by at most
$(\mathsf d+1)\dim A$ elements.
\end{corollary}
\begin{proof} If $\mathsf M^{(n)}$ denotes the $n$-th layer of $\mathsf M/\mathsf K$, then
$\coh^1(\mathsf M/\mathsf K, A(\mathsf M))_{p^\nu}$ is the injective limits of
$\coh^1(\mathsf M^{(n)}/\mathsf K, A(\mathsf M))_{p^\nu}$. The last statement is from Nakayama's lemma.
\end{proof}

\subsection{The order of $\mathfrak H_{w,n}$}\label{su:hnw} Let $v$ be a place of $K$. 
Let $\mathfrak H_{w,n}$, $w\mid v$, be as in \S\ref{su:int}.
\begin{lemma}\label{l:o} 
If $K'/K$ is unramified at $v$ or $A$ has an ordinary reduction at $v$, then $|\mathfrak H_{w,n}|$ is bounded for all $n$ and all places $w$ of $K^{(n)}$ sitting above $v$.
\end{lemma}
\begin{proof} If $A$ has ordinary reduction at $v$, then Lemma \ref{l:fin} says
$$|\mathfrak H_{w,n}|=|\coh^1(G_w, A(K'^{(n)}_{w'}))|=|\coh^1(G_w, A(K'^{(n)}_{w'}))_p|
\leq p^{\dim A+1}.$$

If $A$ has non-ordinary reduction at $v$, then $K^{(n)}_{w}/K_v$ is unramified and so is
$K'^{(n)}_{w'}/K_v$. Over the extension $K'^{(n)}_{w'}/K_v$, the N$\acute{\text{e}}$ron model of $A/K_v$ is stable, thus the group of components $\Pi$ remain the same. By Lemma \ref{l:unrambound},
$$|\mathfrak H_{w,n}|\leq |\Pi_{K^{(n)}_{w}}|\leq |\Pi|.$$

\end{proof}

\begin{lemma}\label{l:deltav}Let $v$ be a place of $K$. The following holds:
\begin{enumerate}
\item[(a)] If $v$ splits completely over $L$, then 
$\log_p|\mathfrak H_{v}^{(n)}|=p^{nd}\cdot \log_p|\coh^1(G_{v'},A(K'_{v'})|.$
\item[(b)] If $\Gamma_v\not=0$ and $K'/K$ is unramified at $v$, then $\log_p|\mathfrak H_{v}^{(n)}|=\mathrm{O}(p^{n(d-1)})$.
\item[(c)] If $\Gamma_v\not=0$ and $A$ has ordinary reduction at $v$, then $\log_p|\mathfrak H_{v}^{(n)}|=\mathrm{O}(p^{n(d-1)})$.
\end{enumerate}
\end{lemma}
\begin{proof}
If $\Gamma_v=0$, then $\mathfrak H_{w,n}=\coh^1(G_v,A(K'_{v'}))$ and the number of $w$ of $K^{(n)}$ sitting above $v$ equals $p^{dn}$; otherwise, the number of those
$w$ is $\mathrm{O}(p^{(d-1)n})$. 

\end{proof}

\subsection{The proof of Theorem B}\label{su:asympt} 
Recall that $\mathfrak H_v^{(n)}=\bigoplus_{w\mid v} \mathfrak H_{w,n}$. 
If some place of $K^{(n)}$ sitting above $v$ splits completely over $K'^{(n)}$, then all places of $K^{(n)}$ sitting above $v$ split completely (since the decomposition subgroups are conjugate to each other), and the same holds for all $m\geq n$.
Then it follows that $\mathfrak H_v^{(n)}=0$ and Theorem B holds trivially. Thus for proving the theorem, we may assume that
for every $n$ and for every place $w$ of $K^{(n)}$ sitting above $v$, the group $G_w$ is non-trivial and can be identified with $G$.

By the Galois action, each 
$\sigma\in\Gal(L'/K)$ induces an isomorphism
$K'^{(n)}_{w'}\simeq K'^{(n)}_{\tensor[^\sigma]{w}{}'}$, 
and therefore an isomorphism
$\coh^1(G_w, A(K'^{(n)}_{w'}))\simeq\coh^1(G_{\tensor[^\sigma]w{}}, A(K'^{(n)}_{\tensor[^\sigma]w{}'}))$
(since $\Gal(L'/K)$ is abelian, $G_{\tensor[^\sigma]w{}}=G_{w}$).
Thus, there is a natural action of $\Gal(L'/K)$ on $\mathfrak H_v^{(n)}$. Actually, the action factors through $\Gamma$, 
because $G_w$ acts trivially on $\coh^1(G_w, A(K'^{(n)}_{w'}))$. 

Similarly, for a place $w$ of $K^{(n)}$, sitting over $v$, and for $m\geq n$, we have a natural 
$\Gamma_n$-module structure of $\mathfrak H_w^{(m)}:=\bigoplus_{u\mid w}\mathfrak H_{u,m}$.
Define the restriction map  
$$\rho_{w,n}^m:\mathfrak H_{w,n}\longrightarrow (\mathfrak H_w^{(m)})^{\Gamma_n}$$
of which each component (will be also called the restriction map)
is the composition 
$$\mathfrak H_{w,n}\longrightarrow \coh^1(K'^{(m)}_{u'}/K^{(n)}_w, A(K'^{(m)}_{u'}))
\longrightarrow \mathfrak H_{u,m}^{\Gamma_{n,w}}$$
of the inflation and the restriction. Using $\rho_{w,n}^m$ as transition maps, define $\mathfrak H_w=\varinjlim_m \mathfrak H_w^{(m)}$ and let $\rho_{w,n}:\mathfrak H_{w,n}\longrightarrow \mathfrak H_w^{\Gamma_n}$ be the injective limit of $\rho_{w,n}^m$ whose component is the restriction map 
$${\rm res}_{u,n}:\mathfrak H_{w,n}\longrightarrow \coh^1(G_u,A(L'_{u'}))^{\Gamma_{n,w}}.$$

Now, $\mathfrak H_v=\varinjlim_m \mathfrak H_v^{(m)}$. For every $n$, we have
$\mathfrak H_v=\bigoplus_{w\mid v} \mathfrak H_w$, where $w$ runs over all places of $K^{(n)}$ sitting over $v$. 
Put $\rho_n=\oplus_{w\mid v}\rho_{w,n}:\mathfrak H_v^{(n)}\longrightarrow \mathfrak H_v^{\Gamma_n}$.

\begin{lemma}\label{l:contH} As $n$ varies, 
$$\log_p|\ker(\rho_n)|=\log_p|\coker (\rho_n)|+\mathrm{O}(p^{n(d-1)})=\mathrm{O}(p^{n(d-1)}).$$
\end{lemma}
\begin{proof} Let $w$ be a place of $K^{(n)}$. An element $\xi=(\xi_u)_{u\mid w} \in\mathfrak H_w$ is fixed by $\Gamma_n$ if and only if for a place $u$ of $L$ (and hence for all $u$), we have $\xi_u\in \coh^1(G_u,A(L'_{u'}))^{\Gamma_{n,w}}$ and 
$\xi_{\tensor[^\sigma]u{}}=\tensor[^\sigma]\xi{_u}$, for all $\sigma\in\Gamma_n$.
Thus, $\ker(\rho_{w,n})$ and $\coker(\rho_{w,n})$ can be identified with the kernel and cokernel of the afore mentioned homomorphism ${\rm res}_{u,n}$.

If $\Gamma_{n,w}=0$, then $L_u=K^{(n)}_w$, so ${\rm res}_{u,n}$ is an isomorphism. Since $\rho_n=\bigoplus_{w\mid v} \rho_{w,n}$, we have $\ker(\rho_n)=\coker(\rho_n)=0$, and the lemma is proved.

Suppose $\Gamma_{n,w}\not=0$. Then there are {$\mathrm{O}(p^{n(d-1)})$} amount of places $w$ of $K^{(n)}$ sitting over $v$, so it is remained to show that $\ker({\rm res}_{u,n})$ and $\coker({\rm res}_{u,n})$ are of bounded orders. 
This is the case, if $A$ has ordinary reduction, because the domain and the target of ${\rm res}_n$
are of bounded orders by Lemma \ref{l:fin}.

If $A$ has non-ordinary reduction, then $L_u/K_w$ is unramified. Decompose ${\rm res}_n$ into
$$\xymatrix{\coh^1(G_{w},A(K'^{(n)}_{w'}))\ar[r]^-\alpha & \coh^1(L'_{u'}/K^{(n)}_w, A(L'_{u'}))\ar[r]^-\beta & \coh^1(G_u, A(L'_{u'}))^{\Gamma_{n,w}}}$$
where $\alpha$ is the inflation (an injection) and $\beta$ the restriction. Now by the inflation-restriction sequence, 
$$\coker(\alpha)\subset \coh^1(L'_{u'}/K'^{(n)}_{w'}, A(L'_{u'})),$$
while the Hochschild-Serre spectral sequence
says 
$$\ker(\beta)=\coh^1(L_u/K^{(n)}_w,A(L_u))\; \text{and}\; \coker(\beta)\subset \coh^2(L_u/K^{(n)}_w,A(L_u)),$$
so by Corollary \ref{c:unrambound}, these are all of bounded orders.

\end{proof}

\begin{lemma}\label{l:cofg} As a $\Lambda_\Gamma$-module, $\mathfrak H_v^\vee$ is finitely generated.
\end{lemma}
\begin{proof} Lemma \ref{l:F'w'} says $\coh^1(G_v, A(K'_{v'}))$ is finite, and hence, by Lemma \ref{l:contH},
so is $\coh^1(G_u, A(L'_{u'}))^{\Gamma_v}$, and hence $\mathfrak H_v^\Gamma$ is finite. If $\mathfrak m$ denotes the maximal ideal of $\Lambda_\Gamma$,
the duality implies that the module $\mathfrak H_v^\vee/\mathfrak m$ is also finite. Then the claimed result follows from Nakayama's lemma.
\end{proof}

\subsubsection{The proof}\label{ss:pfp1}
Since $\mathfrak H_v^\vee$ is annihilated by $p$, it is finitely generated and torsion. Let $\delta_v$ denote its $(p)$-rank. We have to establish the asymptotic formula for $\log_p|\mathfrak H_v^{(n)}|$ in terms of $\delta_v$.

Let $I_n$ denote the kernel of the $\Z_p$-algebra homomorphism $\Lambda_\Gamma\longrightarrow \Z_p[\Gamma^{(n)}]$ induced by the projection $\Gamma\longrightarrow \Gamma^{(n)}$.
 
\begin{lemma}\label{l:asymp} For a given $f\in\Lambda_\Gamma$, $p\nmid f$, and a positive integer $\nu$, asymptotically, 
$$\log_p|\Lambda_\Gamma/(p^\nu, f, I_n)|=\mathrm{O}(p^{n(d-1)}).$$
\end{lemma}
\begin{proof}
Let $\sigma_1,...,\sigma_d$ be a set of topological generators of $\Gamma$ so that $\Lambda_\Gamma=\Z_p[[X_1,...,X_d]]$, $X_i=\sigma_i-1$. By \cite[VII.7 Lemma 3, VII.8 Proposition 5]{bou}, we
can find a change of varables $Y_i=X_i+X_n^{u(i)}$, $i=1,...,d-1$, $Y_n=X_n$, such that the quotient $\Lambda_\Gamma/(f)$ is covered by 
the image of $\sum_{j=0}^{s-1}Y_d^j\cdot \Z_p[[Y_1,...,Y_{d-1}]]$, for some $s$. Since 
$$\Lambda_\Gamma/(p,I_n)=\F_p[X_1,...,X_d]/(X_1^{p^n},...,X_d^{p^n})=\F_p[Y_1,...,Y_d]/(Y_1^{p^n},...,Y_d^{p^n}),$$
as a $\F_p$-module, $\Lambda_\Gamma/(p,f, I_n)$ is generated by (the image of)
$$\{Y_d^j\cdot \prod_{i=1}^{d-1}Y_i^{m_i}\;\mid\; j=0,1,...,s-1; m_i=0,...,p^n-1\},$$
which thus generates the $\Z_p$-module $\Lambda_\Gamma/(f, I_n)$ by Nakayama's lemma. This shows 
$$\log_p |\Lambda_\Gamma/(p^\nu,f, I_n)|\leq s\nu+\nu\cdot p^{n(d-1)}.$$
\end{proof}

\begin{corollary}\label{c:asymp}If $N$ is pseudo-null over $\Lambda_\Gamma$, then for a given positive integer
$\nu$,
$$ \log_p|N/(p^\nu,I_n)|=\mathrm{O}(p^{n(d-1)}). $$
\end{corollary}
\begin{proof}There is a surjective homomorphism $(\Lambda_\Gamma/(f))^r\longrightarrow N$, for some
$f$, $p\nmid f$. Tensoring both sides of the arrow with $\Lambda_\Gamma/(p^\nu, I_n)$ and then apply Lemma \ref{l:asymp} to it.

\end{proof}

\begin{corollary}\label{c:asymp(0)}
If a finitely generated $\Lambda_\Gamma$-module $Z$ has $(p)$-rank $m$, then
$$\log_p|Z/(p,I_n)|=m\cdot p^{nd}+\mathrm{O}(p^{n(d-1)}).$$
\end{corollary}
\begin{proof}
Since $Z/(p)$ and $(\Lambda_\Gamma/(p))^{m}$ are pseudo-isomorphic, there is an exact sequence $(\Lambda_\Gamma/(p))^{m}\longrightarrow Z/(p)\longrightarrow N$, where $N$ is pseudo-null, annihilated by $p$. By tensoring the exact sequence with $\Lambda_\Gamma/(p,I_n)$, we obtain 
$$m\cdot p^{nd}+\mathrm{O}(p^{n(d-1)})\geq \log_p|Z/(p,I_n)|.$$

The inequality in the other direction follows from the symmetry of pseudo-isomorphism. 
\end{proof}

Since $p\cdot\mathfrak H_v=0$, Lemma \ref{l:contH} and the above corollary imply that
$$\log_p|\mathfrak H_v^{(n)}|=\log_p|\mathfrak H_v^{\Gamma_n}|+\mathrm{O}(p^{n(d-1)})=\log_p|\mathfrak H_v^\vee/(p,I_n)|+\mathrm{O}(p^{n(d-1)})=p^{nd}\cdot\delta_v+\mathrm{O}(p^{n(d-1)}).$$ 
This proves Theorem B.

\section{Theorem C}\label{s:tb} In this section, we complete the proof of Theorem C,
in which the first part will be proved in \S\ref{su:dag} (see \eqref{e:dagmldelta} and \eqref{e:ineine}),
while the rest is proved in \S\ref{su:pftb} (see \eqref{e:deltadag} and \eqref{e:gimel}).

\subsection{Control lemmas}\label{su:cont}In this subsection, we discuss control lemmas of various types.
Let us begin with the following simple lemma.
\begin{lemma}\label{l:(1)} For a fixed natural number $\nu$, the order of 
$\coh^i(\Gamma_n, A_{p^\nu}(L))$, $i=1,2$, is bounded as $n$ varies. 
\end{lemma}
\begin{proof} Choose $n_0$ sufficiently large such that $M:=A_{p^\nu}(L)=A_{p^\nu}(K^{(n_0)})$.
If $n\geq n_0$, then by \cite[Lemma 3.2.1]{tan10}, for all positive integer $m$, the order of 
$\coh^i(\Gamma_n/\Gamma_{n+m},M)$ is bounded by $|M|^{d^i}$, hence so 
is the order of the injective limit $\coh^i(\Gamma_n,M)$.
\end{proof}

The above lemma holds for every $\Z_p^d$-extension. From here onwards, we return to studying a $\Z_p^d$-extension $L/K$ that is
only ramified at a finite set of places where $A$ has ordinary reductions.
\begin{lemma}\label{l:cont}
For each natural number $\nu$, the restriction map
$$\Sel_{p^\nu}(A/K^{(n)})\longrightarrow \Sel_{p^\nu}(A/L)^{\Gamma_n}$$
has kernel and cokernel of bounded orders as $n$ varies. 
\end{lemma}
\begin{proof} The lemma follows from the commutative diagram of exact sequences
$$\xymatrix{\ker_n\ar[r] & \coh^1(K^{(n)},A_{p^\nu})\ar[r]^-{\mathrm{res}_{L/K^{(n)}}} &  \coh^1(L,A_{p^\nu})^{\Gamma_n} \ar[r] & \coker_n\\
 & \mathrm{res}_{L/K^{(n)}}^{-1}(\Sel_{p^\nu}(A/L)^{\Gamma_n}) \ar[r] \ar@{^(->}[u] & \Sel_{p^\nu}(A/L)^{\Gamma_n} \ar@{^(->}[u] &}
 $$
 and the exact sequence
$$\xymatrix{ \Sel_{p^\nu}(A/K^{(n)}) \ar@{^(->}[r]  & \mathrm{res}_{L/K^{(n)}}^{-1}(\Sel_{p^\nu}(A/L)^{\Gamma_n}) \ar[r]  &\bigoplus_w \coh^1(L_w/K_w^{(n)},A)_{p^\nu}}.$$ 

In view of Lemma \ref{l:(1)}, Hochschild-Serre spectral sequence implies that
the orders of $\ker_n$ and $\coker_n$ are bounded. Also, the group 
$\bigoplus_w \coh^1(L_w/K_w^{(n)},A)_{p^\nu}$ has bounded order,
by Corollary \ref{c:unrambound} and \eqref{e:unibd}. 
\end{proof}
\begin{corollary}\label{c:cofin} The dual Selmer group $X_L$ is finitely generated over $\Lambda_\Gamma$.
\end{corollary}

\begin{proof} Lemma \ref{l:cont} implies the finiteness of $\Sel_{p}(A/L)^\Gamma$ which is the Pontryagin dual of
$X_L/\mathfrak m\cdot X_L$, where $\mathfrak m$ denotes the maximal ideal of $\Lambda_\Gamma$.
Then apply Nakayama's lemma to dedue the result.

\end{proof}
\begin{corollary}\label{c:asymp(1)} Asymptotically, $\log_p|\Sel_p(A/K^{(n)})|=m_L\cdot p^{nd}+\mathrm{O}(p^{n(d-1)})$. 
\end{corollary}
\begin{proof} Lemma \ref{l:cont} says $\log_p|\Sel_p(A/K^{(n)})|=\log_p|\Sel_p(A/L)^{\Gamma_n}|+\mathrm{O}(1)$. 
Observe that $\Sel_p(A/L)^{\Gamma_n}$ is the Pontryagin dual of $X_L/(p,I_n)$ and then apply Corollary \ref{c:asymp(0)} to deduce the claimed result.
\end{proof}




\begin{lemma}\label{l:(2)} The kernel and cokernel of the restriction
$$\mathrm{res}_{L/K^{(n)}}:  \coh^1(K^{(n)},A_{p^\infty})\longrightarrow  \coh^1(L,A_{p^\infty})^{\Gamma_n}$$
are finite.
\end{lemma}
\begin{proof}As discussed in \cite[\S1.1.6]{tan13}, this follows from
\cite[Proposition 3.3]{grn03} or \cite[Corollary 3.2.4]{tan10} (for the characteristic $p$ case, the proof
of [Ibid, Lemma 3.1.1] still works since a place where $A$ has non-split multiplicative reduction becomes a place
of split multiplicative reduction over some unramified extension).
\end{proof}

\begin{corollary}\label{c:(2)} Let $r_n$ denote the $\Z_p$-corank of $\Sel_{p^\infty}(A/K^{(n)})$. If $X_L$ is torsion over $\Lambda_\Gamma$, then $r_n=\mathrm{O}(p^{n(d-1)})$.
\end{corollary}
\begin{proof}Let $s_n$ be the $\Z_p$-corank of $\Sel_{p^\infty}(A/L)^{\Gamma_n}$. By \cite[Theorem 9]{tan13}, we have $s_n=O(p^{n(d-1)})$. In view of Lemma \ref{l:(2)}, using an argument similar to that in the proof of Lemma \ref{l:cont}, we obtain $s_n= r_n+{\mathrm{O}(1)}$. The claimed result thus follows.
\end{proof}

\begin{lemma}\label{l:nu} Suppose $X_L$ is torsion with elementary $\mu$-invariants $p^{\alpha_1},...,p^{\alpha_m}$.
For each positive integer $\nu$, asymptotically
\begin{equation}\label{e:sha}
\begin{array}{rcl}
p^{nd}\cdot \sum_{i=1}^{m_L} \mathrm{min}(\nu,\alpha_i) & = & \log_p |\Sel_{p^\nu}(A/K^{(n)})|+\mathrm{O}(p^{n(d-1)})\\
&=& \log_p |\overline\Sha_{p^\nu}(A/K^{(n)})|+\mathrm{O}(p^{n(d-1)}).\\
 \end{array}
 \end{equation}
\end{lemma}
\begin{proof} Take $Z=X_L$ in \eqref{e:iwasawa}. Lemma \ref{l:asymp} says 
$\log_p|\Lambda_\Gamma/(p^\nu,\eta_j^{\beta_j},I_n)|$ and $\log_p|N/(p^\nu,I_n)|$ are both 
asymptotically equal to $\mathrm{O}(p^{n(d-1)})$. An argument similar to that in the proof of Corollary \ref{c:asymp(0)}
shows that
the left-hand side of \eqref{e:sha} equals $
\log_p|\Sel_{p^\nu}(A/L)^{\Gamma_n}|+\mathrm{O}(p^{n(d-1)})$, so the first equality follows from 
Lemma \ref{l:cont}. Apply the multiplication by $p^\nu$ to the exact sequence 
$\xymatrix{
\Sel_{\mathrm{div}}(A/K^{(n)})\ar@{^(->}[r] & \Sel_{p^\infty}(A/K^{(n)}) \ar@{->>}[r] & \overline\Sha(A/K^{(n)}) }$ and obtain via the snake lemma the exact sequence 
\begin{equation}\label{e:overline}
\xymatrix{(
\Sel_{\mathrm{div}}(A/K^{(n)})_{p^\nu} \ar@{^(->}[r] & \Sel_{p^\nu}(A/K^{(n)}) \ar@{->>}[r] & \overline\Sha_{p^\nu}(A/K^{(n)}) }.
\end{equation}
The second equality thus follows from Corollary \ref{c:(2)}. 
\end{proof}


\begin{lemma}\label{l:contG}
For each $\nu$, $\nu=\infty$ allowed, the kernel and the cokernel of the restriction
$$\xymatrix{\coh^1(K^{(n)},A_{p^\nu}) \ar[r]^-{\mathfrak {res}_n}  & \coh^1(K'^{(n)},A_{p^\nu})^{G}}$$
are finite of bounded orders as $n$ varies. 
\end{lemma} 
\begin{proof} Let $M=A_{p^\nu}(K'^{(n)})$, which is spanned by at most $\mathsf r=2\dim A$ elements.
The group of $M$-valued inhomogeneous $1$-cocycles is spanned by at most $\mathsf r$ elements, hence so is its quotient 
$\coh^1(G, M)$. Since $\coh^2(G,M)=M^G/\Nm_G(M)$ is a sub-quotient of $M$, it is also spanned by at most
$\mathsf r$
elements. Because $p=|G|$ annihilates the cohomology groups, $\log_p|\coh^i(G, M)|\leq \mathsf r$, $i=1,2$. 
To see this holds for $\nu=\infty$, take the injective limit over $\nu$.
At last, apply Hochschild-Serre spectral sequence.
\end{proof}

It follows from the previous two lemmas that 
\begin{equation}\label{e:mumu'}
\mu_{L'/K'}\geq \mu_{L/K}.
\end{equation}

\subsection{The $\dag$-invariant}\label{su:dag} 
Let $Y_{L'}$ denote the Pontryagin dual of $\Sel_p(A/L')^G$, so that $\dag$ equals its $(p)$-rank.
Since $Y_{L'}/I_n$ is the dual of $\Sel_p(A/L')^{\Gal(L'/K^{(n)})}$, Corollary \ref{c:asymp(0)} says 
\begin{equation}\label{e:dag}
p^{nd}\cdot \dag=\log_p|\Sel_p(A/L')^{\Gal(L'/K^{(n)})}|+\mathrm{O}(p^{n(d-1)}),\;\text{as $n$ varies}.
\end{equation}

\begin{lemma}\label{l:dag}The kernel and cokernel of the restriction map
$$\Sel_p(A/K'^{(n)})^G \longrightarrow \Sel_p(A/L')^{\Gal(L'/K^{(n)})}$$
are of bounded orders as $n$ varies.
\end{lemma}
\begin{proof}By Lemma \ref{l:cont}, the $\mathrm{ker}_n$ and $\mathrm{coker}_n$ in the exact sequences
$$\xymatrix{0\ar[r] & \mathrm{ker}_n\ar[r] &  \Sel_p(A/K'^{(n)}) \ar[r] & \mathrm{im}_n\ar[r] &   0,}$$
and
$$\xymatrix{0\ar[r] & \mathrm{im}_n \ar[r] & \Sel_p(A/L')^{\Gamma_n} \ar[r] & \mathrm{coker}_n\ar[r] & 0} $$
are of bounded orders. Let $\tau$ be a generator of $G$.
 Multiply both sequences by $\tau-1$ and apply the snake lemma to obtain
$$\xymatrix{0\ar[r] & \mathrm{ker}_n^G\ar[r] &  \Sel_p(A/K'^{(n)})^G \ar[r] & \mathrm{im}_n^G\ar[r] & \mathrm{ker}_n/(\tau-1)\mathrm{ker}_n\ar[r] & \cdots,}$$
and
$$\xymatrix{0\ar[r] & \mathrm{im}_n^G \ar[r] & \Sel_p(A/L')^{\Gal(L'/K^{(n)})} \ar[r] & \mathrm{coker}_n^G\ar[r] &\cdots. } $$
\end{proof}

\begin{corollary}\label{c:dag} Asymptotically, 
$$p^{nd}\cdot\dag=\log_p|\Sel_p(A/K'^{(n)})^G|+\mathrm{O}(p^{n(d-1)}).$$
If $X_{L'}$ is torsion over $\Lambda_\Gamma$, then also
$$p^{nd}\cdot\dag
=\log_p|\overline\Sha_p(A/K'^{(n)})^G|+\mathrm{O}(p^{n(d-1)}).$$
\end{corollary}

\begin{proof}The first assertion is due to the Lemma \ref{l:dag} and equation \eqref{e:dag}. For the second assertion,  set
the exact sequence
$$\xymatrix{(
\Sel_{\mathrm{div}}(A/K'^{(n)})_{p} \ar@{^(->}[r] & \Sel_{p}(A/K'^{(n)}) \ar@{->>}[r] & 
\overline\Sha_{p}(A/K'^{(n)}) },$$
which is similar to \eqref{e:overline}. Multiply $\tau-1$ to it and apply the snake lemma would yield the result.
\end{proof}

Lemma \ref{l:contG} (for $\nu=1$) says $|\Sel_p(A/K'^{(n)})^G|=|\mathfrak{res}_n^{-1}(\Sel_p(A/K'^{(n)})^G)|+\mathrm{O}(1)$. The localization map $\mathfrak L_n:\coh^1(K^{(n)},A_{p})\longrightarrow \bigoplus_w\coh^1(K_w^{(n)},A)$ gives rise to 
the exact sequence
\begin{equation}\label{e:selsh}
\xymatrix{ 0\ar[r] &\Sel_p(A/K^{(n)}) \ar[r] & \mathfrak{res}_n^{-1}(\Sel_p(A/K'^{(n)})^G)\ar[r]^-{\mathfrak L_n} & \mathfrak H_n}.
\end{equation}
These together with Equation \eqref{e:asympt}, Corollary \ref{c:asymp(1)}, and Corollary \ref{c:dag} imply
\begin{equation}\label{e:dagmldelta}
m_L+\delta\geq \dag \geq m_L.
\end{equation}
\subsubsection{The Jordan canonical forms}\label{ss:jdc}
For a variable $x$, over $\Z$, we have
$$1+x+\cdots+x^{p-1}=\frac{x^p-1}{x-1}=\frac{((x-1)+1)^p-1}{x-1}\equiv (x-1)^{p-1}\pmod{p}.$$
Therefore, over $\F_p[G]$, for a generator $\tau$ of $G$, the trace
 \begin{equation}\label{e:normG}
 \Nm_G=1+\tau+\cdots+\tau^{p-1}=(\tau-1)^{p-1}.
 \end{equation}
 Suppose $G$ acts linearly on a $\F_p$-vector space $V$. Since $\tau^p=1$,  
 the action of $\tau$ has a unique eigenvalue $1$, and hence
 the number of its Jordan blocks equals $\log_p|V^G|$. Since $(\tau-1)^p=0$, every Jordan block of $\tau$ has the size at most
$p$. Let $\kappa$ denote the number of Jordan blocks having size $p$. Then 
\begin{equation*}\label{e:nmtau}
\kappa=\log_p|(\tau-1)^{p-1}(V)|=\log_p|\mathrm N_G(V)|
\end{equation*}
where the second equality follows from \eqref{e:normG}. Thus,
\begin{equation}\label{e:ml'}\log_p|\mathrm N_G(V)|\cdot p+(\log_p|V^G|-\log_p|\mathrm N_G(V)|)(p-1)\geq \log_p|V|\geq \log_p|V^G|.
\end{equation}
Take $V=\Sel_p(A/K'^{(n)})$. Then $\Nm_G=\mathfrak {res}_n\circ \mathfrak {cor}_n$, where $\mathfrak {cor}_n$
denotes the co-restriction map $V\longrightarrow \Sel_p(A/K^{(n)})$, so by Lemma \ref{l:contG}, $\log_p|\mathrm N_G(V)|\leq p^n\cdot m_L+\mathrm{O}(p^{n(d-1)})$.  
Since $m_{L'}$ is the $(p)$-rank of $X_{L'}$, Lemma \ref{l:nu} says
\begin{equation*}
\log_p|V|=p^n\cdot m_{L'}+\mathrm{O}(p^{n(d-1)}).
\end{equation*}
Also, and
$ \log_p|V^G|=p^n\cdot \dag+\mathrm{O}(p^{n(d-1)})$ by equation \eqref{e:sha} and Corollary \ref{c:dag}. Hence,
by \eqref{e:ml'},
\begin{equation}\label{e:ineine}
(p-1) \dag+m_L\geq m_{L'}\geq \dag.
\end{equation}

\subsection{The proof}\label{su:pftb} In view of the inequalities \eqref{e:dagmldelta} and \eqref{e:ineine},
the first part of {Theorem C} is proven. For proving the remaining assertions of Theorem C, we shall (as we can)
{\em{assume that $X_{L'}$ is torsion over $\Lambda_\Gamma$}}. Then the following control lemma holds.

\begin{lemma}\label{l:contGA} Let $\mathsf{ker}_n$ and $\mathsf{coker}_n$ denote the kernel and cokernel of the restriction map
$$\mathsf{res}_n:\coh^1(K^{(n)}, A) \longrightarrow \coh^1(K'^{(n)},A)^G.$$
Then $\log_p|\mathsf{ker}_n|=\mathrm{O}(p^{n(d-1)})=\log_p|\mathsf{coker}_n|$, as $n$ varies.
\end{lemma}
\begin{proof} As in the proof of Lemma \ref{l:contG}, let $M=A(K'^{(n)})$. Since $|G|=p$ and $M$ is spanned by
at most
$\mathrm{O}(p^{n(d-1)})$ elements, $\log_p|\coh^1(G,M)|=\mathrm{O}(p^{n(d-1)})$.
Since $X_L$ is also torsion, $M^G=A(K^{(n)})$ is spanned by at most
$\mathrm{O}(p^{n(d-1)})$ elements, so is $\coh^2(G,M)=M^G/\Nm_G(M)$.
\end{proof}

The lemma below is from \cite[Theorem 2]{lltt18}.

\begin{lemma}\label{l:aa'} The dual Selmer group $X^t_L$ of $A^t/L$ is also torsion, and has the same elementary $\mu$-invariant as those of $X_L$.
\end{lemma}
 
\subsubsection{The Cassels-Tate dual exact sequence}\label{ss:ctexact} Denote
$$\mathsf S^{(n)}:=\mathsf {res}_n^{-1}(\Sha_{p^\infty}(A/K'^{(n)})^G).$$
In view of Lemma \ref{l:contGA}, we see that
\begin{equation}\label{e:cokerGA}
\log_p|\coker(\xymatrix{\mathsf S^{(n)}\ar[r]^-{\mathsf {res}_n} & \Sha_{p^\infty}(A/K'^{(n)})^G})|=\mathrm{O}(p^{n(d-1)})
\end{equation}
Similar to \eqref{e:selsh}, the localization map 
$$\mathsf L_n:\coh^1(K^{(n)},A)\longrightarrow \bigoplus_w\coh^1(K_w^{(n)},A)$$ 
gives rise to the exact sequence
\begin{equation}\label{e:shash}
\xymatrix{ 0\ar[r] &\Sha_{p^\infty}(A/K^{(n)}) \ar[r] & \mathsf S^{(n)}\ar[r]^-{\mathsf L_n} & \mathfrak H_n}.
\end{equation}

\begin{lemma}\label{l:cteact} Let $\mathsf{coker}_n$ denote the co-kernel of
$\xymatrix{\mathsf S^{(n)} \ar[r]^-{\mathsf L_n} & \mathfrak H_n}$.  
Then
$$\log_p|\mathsf{coker}_n|=\mathrm{O}(p^{n(d-1)}).$$
\end{lemma}
\begin{proof} By Corollary \ref{c:(2)}, the $\Z_p$-corank of $\Sel_{p^\infty}(A^t/K^{(n)})$ is $\mathrm{O}(p^{n(d-1)})$.
This implies that the $\Z_p$-rank of $\mathrm T\Sel_{p^\infty}(A^t/K^{(n)})$, the Tate-modules  of $\Sel_{p^\infty}(A^t/K^{(n)})$, is also $\mathrm{O}(p^{n(d-1)})$. 

Global and local Kummer exact sequences induces the homomorphism 
$$\Sel_{p^\infty}(A^t/K^{(n)})\longrightarrow \prod_{\text{all}\;w} A^t(K_w^{(n)})\otimes \Q_p/\Z_p.$$ 
Taking the Tate-modules, one obtains 
$$\mathsf t_n: \mathrm T\Sel_{p^\infty}(A^t/K^{(n)})\longrightarrow \prod_{\text{all}\;w} A^t(K_w^{(n)})^{\widehat{}}.$$
Here the left-hand side is a free $\Z_p$-module whose rank is $\mathrm{O}(p^{n(d-1)})$, while the right-hand side is the $p$-completion of $\prod_{\text{all}\;w} A^t(K_w^{(n)})$ and can be identified with the Pontryagin dual of 
$\bigoplus_{\text{all}\;w} \coh^1(K^{(n)}_w,A)(p)$.
Under such identification, the generalized Cassels-Tate dual exact sequence (see \cite[Main Theorem]{got07})
says $\coker(\mathsf L_n)$ is exactly the Pontryagin dual of $\image(\mathsf t_n)$ which is spanned by at most
$\mathrm{O}(p^{n(d-1)})$ elements.
Since $\mathsf S^{(n)}=\mathsf L_n^{-1}(\mathfrak H_n)$, the natural map $\mathsf{coker}_n\longrightarrow \coker(\mathsf L_n)_p$ is injective, so 
$\log_p|\mathsf{coker}_n|\leq \log_p|\coker(\mathsf L_n)_p|=\mathrm{O}(p^{n(d-1)})$.
\end{proof}


For each $n=1,2,...$, let $\mathsf T^{(n)}$ denote either the group $\mathsf S^{(n)}$ or the group $\Sha_{p^\infty}(A/K^{(n)})$. Since $\mathsf T^{(n)}/\Sha_{p^\infty}(A/K^{(n)})$ is annihilated by $p$, $\Sha_{\mathrm{div}}(A/K^{(n)})$ is the $p$-divisible subgroup of $\mathsf T^{(n)}$. Thus $\overline{\mathsf T}^{(n)}:=\mathsf T^{(n)}/\Sha_{\mathrm{div}}(A/K^{(n)})$ is a finite group.

\begin{lemma}\label{l:overlineK'}
For a given positive integer $\nu$, if 
$$\overline{\mathsf{res}}_n:\overline{\mathsf T}^{(n)}_{p^\nu}\longrightarrow \overline\Sha(A/K'^{(n)})^G_{p^\nu}$$ denote the restriction map, then 
\begin{equation*}
 \log_p|\ker(\overline{\mathsf{res}}_n)|=\mathrm{O}(p^{n(d-1)}).
 \end{equation*}
\end{lemma}
\begin{proof}
Apply the snake lemma to the commutative diagram of exact sequences
\begin{equation}\label{e:overlineK'}
\xymatrix{0\ar[r] & \Sha_{\mathrm{div}}(A/K^{(n)})_{p^\nu} \ar[r] \ar[d]^-{\mathsf a}&  \mathsf T^{(n)}_{p^\nu}
 \ar[r] \ar[d]^-{\mathsf{res}_n}& \overline{ \mathsf T}^{(n)}_{p^\nu}\ar[r] \ar[d]^-{\overline{\mathsf{res}}_n} & 0\\
0\ar[r] & \Sha_{\mathrm{div}}(A/K'^{(n)})_{p^\nu} \ar[r] &  \Sha_{p^\nu}(A/K'^{(n)}) \ar[r] & \overline\Sha_{p^\nu}(A/K'^{(n)})\ar[r] & 0,}
\end{equation}
where the exactness of the rows are by similar reason for that of \eqref{e:overline}.
By Corollary \ref{c:(2)} and Lemma \ref{l:contGA} both $\log_p|\coker(\mathsf a)|$ and $\log_p|\ker(\mathsf{res}_n)|$
equal $\mathrm{O}(p^{n(d-1)})$, and so is $\log_p|\ker(\overline{\mathsf{res}}_n)|$.
\end{proof}

The short exact sequence \eqref{e:shash} induces
\begin{equation}\label{e:overlineshash}
\xymatrix{ 0\ar[r] &\overline\Sha(A/K^{(n)}) \ar[r] & \overline{\mathsf S}^{(n)} \ar[r]^-{\overline{\mathsf L}_n} & 
\mathsf L_n(\mathsf S^{(n)})\ar[r] & 0.}
\end{equation}
Since $\mathfrak H_n$ is annihilated by $p$, $\mathsf L_n(\mathsf S^{(n)})=\mathsf L_n(\mathsf S^{(n)})/p\mathsf L_n(\mathsf S^{(n)})$, so the above induces a surjection from
$\overline{\mathsf S}^{(n)}/p\overline{\mathsf S}^{(n)}$ to $\mathsf L_n(\mathsf S^{(n)})$. Then \eqref{e:asympt} and
Lemma \ref{l:cteact} implies that
\begin{equation}\label{e:overlineS}
\log_p|\overline{\mathsf S}^{(n)}_p|=\log_p|\overline{\mathsf S}^{(n)}/p\overline{\mathsf S}^{(n)}|\geq
p^{nd}\cdot \delta+\mathrm{O}(p^{n(d-1)}).
\end{equation}

Since $\overline{\mathsf{res}}_n(\overline{\mathsf S}_p^{(n)})\subset \overline\Sha(A/K'^{(n)})_p^G$, by Corollary \ref{c:dag} and Lemma \ref{l:overlineK'},
\begin{equation}\label{e:deltadag}
\dag\geq\delta.
\end{equation}

\subsubsection{The Cassel-Tate pairing}\label{ss:ct}
Cassels-Tate pairing induces an alternating perfect pairing
$$ \langle - , - \rangle_{A/K^{(n)}}:\overline\Sha(A/K^{(n)})\times \overline\Sha(A^t/K^{(n)})\longrightarrow \Q_p/\Z_p$$
that fits into the commutative diagram 
\begin{equation*}
\xymatrixcolsep{5pc}\xymatrix{\overline\Sha(A/K'^{(n)})\times \overline\Sha(A^t/K'^{(n)}) \ar@<-5ex>[d]_{\overline{\rm cor}_n} 
\ar[r]^-{\langle - , - \rangle_{A/K'^{(n)}}} & \Q_p/\Z_p\ar@{=}[d]\\
\overline\Sha(A/K^{(n)})\times  \overline\Sha(A^t/K^{(n)}) \ar@<-5ex>[u]_{\overline{\rm res}_n^t} 
\ar[r]^-{\langle - , - \rangle_{A/K^{(n)}}} & \Q_p/\Z_p,}
\end{equation*}
where $\overline{\rm cor}_n$ and $\overline{\rm res}_n^t$ are respectively induced from
the co-restriction and the restriction. 
The diagram shows that via the pairing ${\langle - , - \rangle_{A/K^{(n)}}}$, the Pontryagin dual of
$\coker(\overline{\rm cor}_n)$ is identified with $\ker( \overline{\rm res}_n^t)$. Because 
${\overline{\rm cor}}_n^t\circ {\overline{\rm res}}_n^t$ is the multiplication-by-$p$, 
$\ker (\overline{\rm res}_n^t)$ is annihilated by $p$. 
By Lemma \ref{l:overlineK'} (applied to $A^t$, $\nu=1$),
\begin{equation}\label{e:overlinecoker}
\begin{array}{rcl}
\log_p|\coker(\overline{\rm cor}_n)|&=&\log_p|\ker (\overline{\rm res}_n^t)|\\
&=&\mathrm{O}(p^{n(d-1)}).
\end{array}
\end{equation}

Let $\nu$ be greater than the exponents of elementary $\mu$-invariants of $X_L$ and $X_{L'}$. Lemma \ref{l:nu} implies that
$\log_p|\overline\Sha_{p^{\nu+1}}(A/K'^{(n)})|=\log_p|\overline\Sha_{p^\nu}(A/K'^{(n)})|+\mathrm{O}(p^{n(d-1)})$.
Thus, 
$$\log_p|\overline\Sha_{p^{\nu+1}}(A/K'^{(n)})/\overline\Sha_{p^\nu}(A/K'^{(n)})|=\mathrm{O}(p^{n(d-1)}).$$
The multiplication-by-$p^\nu$ induces an isomorphism from $\overline\Sha_{p^{\nu+1}}(A/K'^{(n)})/\overline\Sha_{p^\nu}(A/K'^{(n)})$ to $p^\nu \overline\Sha(A/K'^{(n)})\cap \overline\Sha_{p}(A/K'^{(n)})$, so
$$\log_p|p^\nu \overline\Sha(A/K'^{(n)})\cap \overline\Sha_{p}(A/K'^{(n)})|=\mathrm{O}(p^{n(d-1)}).$$
The inclusion induces the homomorphism 
$$\gimel_n: \overline\Sha_{p}(A/K'^{(n)})\longrightarrow \overline\Sha(A/K'^{(n)})/p^\nu \overline\Sha(A/K'^{(n)}),$$ whose kernel is actually $p^\nu \overline\Sha(A/K'^{(n)})\cap \overline\Sha_{p}(A/K'^{(n)})$, so by Corollary \ref{c:dag},
Lemma \ref{l:overlineK'} and \eqref{e:cokerGA},
\begin{equation}\label{e:gimel}
\log_p|\gimel_n(\overline{\rm{res}}_n(\overline{\mathsf S}_p^{(n)}))|\geq p^{nd}\cdot\dag+\mathrm{O}(p^{n(d-1)}).
\end{equation}
Consider the exact sequence
$$\xymatrix{\ker(\mathsf c_n) \ar@{^(->}[r] & \overline\Sha(A/K'^{(n)})/p^\nu \overline\Sha(A/K'^{(n)}) \ar[r]^-{\mathsf c_n} & \overline\Sha(A/K^{(n)})/p^\nu \overline\Sha(A/K^{(n)}) \ar@{->>}[r] &\coker(\mathsf c_n),}$$
where $\mathsf c$ is induced from $\overline{\mathrm {cor}}_n$.
Since $\overline{\mathrm {cor}}_n\circ \overline{\mathrm {res}}_n$ is the multiplication-by-$p$,
$$\gimel_n(\overline{\rm{res}}_n(\overline{\mathsf S}_p^{(n)}))\subset \ker(\mathsf c_n).$$
Hence, by Lemma \ref{l:dag}, \eqref{e:overlinecoker} and \eqref{e:gimel},
\begin{equation}\label{e:dagmumu}
\begin{array}{rcl}
\mu_{L'/K'}&=& p^{-nd}\cdot (\log_p|\overline\Sha_{p^\nu}(A/K'^{(n)})|+\mathrm{O}(p^{n(d-1)}))\\
&=&p^{-nd}\cdot (\log_p|\overline\Sha(A/K'^{(n)})/p^\nu \overline\Sha(A/K'^{(n)})|+\mathrm{O}(p^{n(d-1)}))\\
&\geq& \dag+p^{-nd}\cdot (\log_p|\overline\Sha(A/K^{(n)})/p^\nu \overline\Sha(A/K^{(n)})|+\mathrm{O}(p^{n(d-1)}))\\
&=&\dag+\mu_{L/K}.
\end{array}
\end{equation}
This completes the proof of Theorem C.

\section{Examples}\label{s:comp}
In this section, we present numerical evidence of Theorem A, B and C. In \S\ref{su:ch2} and \S\ref{su:eat}, we only consider the unramified $\Z_p$-extension $L/K$, $L=K_p^{(\infty)}$. We shall apply \cite[Theorem 5.1.1]{lst21}, which says
\begin{equation}\label{e:lst}
\mu_{L/K}=\frac{\deg \Delta_{A/K}}{12}+g_K-1-\theta,
\end{equation}
where $A/K$ is taken to be a non-constant elliptic curve, $\Delta_{A/K}$ is the minimal discriminant of $A/K$, $g_K$ is the genus of $K$ and for $P(q^{-s})=L_{A/K}(s)$, which is
the Hasse-Weil $L$-function, $\theta$ the integer such that $q^\theta P(T/q)$ is a primitive polynomial in $\Z[T]$.

Another useful equation is the following, which is from \cite[(27)]{tan21},
\begin{equation}\label{e:tan22}
\sum_v n_v\cdot\deg v=\frac{(p-1)\deg \Delta_{A/K}}{12},
\end{equation}
where $v$ is taken over all supersingular places.

If $K'/K$ is ramified at $v$, thus $K'=K(y)$, where $y^p-y=D$, $\ord_v (D)=-pm-k<0$, for $0< k<p$, 
then (see \cite[Lemma 2.3]{thd20})
\begin{equation}\label{e:lambda}
\lambda_v=pm+k.
\end{equation}

 In order to use Magma\footnote{We take this opportunity to mention a mistake in \cite[\S 7.3]{lst21}. The value $\theta$ calculated there is in fact 1 so that the $\mu$-invariant is 1 instead of 0.} to calculate the $L$-functions, we take $K=\F_p(t)$, in \S\ref{su:ch2} and \S\ref{su:eat}. 
In \S\ref{su:non}, we give an example where $K$ is not the rational field, $L/K$ ramified, such that
$\mu_{L/K}$ is finite and $\mu_{L'/K'}=\infty$ 
 
\subsection{The $p=2$ case}\label{su:ch2} In this case, $k=K$. Also,
the calculation of $L_{A/K'}$ can be 
done by computing the $L$-function of the twist curve.
Indeed, if $K'/K$ is the Artin-Schreier extension defined by the equation $T^2-T=D$ and $A_D$ is the twist of $A$ by the quadratic extension $K'/K$, then
$L_{A/K'}=L_{A/K}\cdot L_{A_D/K}$.

Suppose $A/K$ is defined by Deuring normal form
\begin{equation}\label{e:deuring}
Y^2+\alpha XY+Y=X^3.
\end{equation}      
The substitution         
$$X\mapsto \alpha^2X+1/\alpha,\ Y\mapsto \alpha^3Y+1/\alpha^3$$
transforms the defining equation into
$$Y^2+XY=X^3+1/\alpha^3X^2+1/\alpha^9+1/\alpha^{12}.$$

The lemma below is extracted from \cite[Appendix A Exercise A.2]{sil86}.
\begin{lemma}\label{l:Deuring} Let $A/K$ be an elliptic curve in characteristic 2 defined by 
        $$Y^2+XY=X^3+a_2X^2+a_6.$$
Then the quadratic twist $A_D$ is defined by the equation
$$Y^2+XY=X^3+(a_2+D) X^2+a_6.$$

\end{lemma}
\begin{corollary}\label{c:deuring}
If $A/K$ is defined by \eqref{e:deuring}, then $A_D$ is defined by  
$$Y^2+\alpha XY+Y=X^3+D\alpha^2 X^2+D.$$
\end{corollary}

\subsubsection{The curve $A$}\label{ss:I}

Consider the elliptic curve over $\F_2(t)$
$$A:Y^2+tXY+Y=X^3.
$$
Its arithmetic invariants are as follows:
\begin{enumerate}
\item $\mu_{L/K}=0$.
\item $\Delta_{A/K}=v_1+v_\omega+9 v_\infty$, where $v_1,v_\omega,v_\infty$
are respectively the zero of $t-1$, $t^2+t+1$, $1/t$. 
It has multiplicative reduction at these three places, so $A/K$ has semi-stable reduction everywhere.
The j-invariant equals $\frac{t^{12}}{t^3+1}$.
\item The place $v_0$, the zero of $t$, is the only supersingular place of $A$. By \eqref{e:tan22}, $n_{v_0}=1$.
\end{enumerate} 

In fact, the only supersingular curve in characteristic $2$ is the one with j-invariant $0$ (\cite[Exercise 5.7]{sil86}).

\subsubsection{The extension $K'/K$}\label{ss:K'}
Let $K'/K$ be a quadratic extension only ramified at $v_0$, with
$\lambda_{v_0}=1$. Lemma \ref{l:small} says $\overline{\mathrm b}_v=\underline{\mathrm b}_v=0$,
hence $\delta_{v_0}=0$, and by Lemma \ref{l:deltav}, $\delta=0$.
Because $\mu_{L/K}=0$, Theorem C says $\mu_{L'/K'}=0$.

Such $K'/K$ is the Artin-Schreier extension defined by the equation $T^2-T=1/t+\alpha $, $\alpha\in\F_2$. 
The computation also outputs $\mu_{L'/K'}=0$.

\subsubsection{The extension $K''/K$}
Let $K''/K$ be a quadratic extension only ramified at $v_0$, with
$\lambda_{v_0}=3$. It is 
the Artin-Schreier extension defined by the equation
$T^2-T=1/t^3+\alpha/t+\beta$, $\alpha,\beta\in\F_2$. 
The computation outputs $\mu_{L''/K''}=2$.

We have $\overline{\mathrm b}_v=2$ and $\underline{\mathrm b}_v=1$, so by Corollary \ref{c:a},
$2\geq \delta_{v_0}\geq 1$, hence $\delta=\delta_{v_0}=1$ or $2$.
Since $m_{L/K}=0$, $\dag=\delta$. Theorem C says $2=\mu_{L''/K''}\geq\dag$ and $m_{L''}=\dag$. There are two possibility for the elementary $\mu$-invariants of $X_{L''}$, they are either $p$, $p$, or $p^2$.


\subsubsection{The Frobenius twist $A^{(2)}$}\label{ss:ftwist}
Consider the Frobenius twist of $A$,
$$A^{(2)}:Y^2+t^2XY+Y=X^3.
$$
The discriminant of $A^{(2)}/K$ is $24$. Since $A^{(2)}$ is isogenous to $A$, $v_0$ is the only supersingular place for $A^{(2)}/K$, and, by \eqref{e:tan22}, $n^{(2)}_{v_0}=2$.

Let $\mu_{L/K}^{(2)}$, $\mu_{L'/K'}^{(2)}$, and $\mu_{L''/K''}^{(2)}$ denote the the $\mu$-invariants associated to $A^{(2)}$. Computation output says $\mu_{L/K}^{(2)}=1$, $\mu_{L'/K'}^{(2)}=2$, and $\mu_{L''/K''}^{(2)}=4$,
the outcome is consistent with \cite[Proposition 4]{tan21}, which also says, as the discriminant of $A$ over $K$, $K'$, and $K''$ are respectively $12$, $24$, $24$, the elementary $\mu$-invariants of $A^{(2)}$ are
$p$ for $L/K$, $p$, $p$ for $L'/K'$, and either $p^2$, $p^2$ or $p$, $p^3$, for $L''/K''$. Thus, $m^{(2)}_L=1$,
$m^{(2)}_{L'}=2$ and $m^{(2)}_{L''}=2$.

For $A^{(2)}/L'$, Lemma \ref{l:small} says $\delta=0$, while by Theorem C, $\dag=m^{(2)}_L=1$,
$2\geq m^{(2)}_{L'}\geq1$.

For $A^{(2)}/L''$, $\overline{\mathrm b}_v=2$ and $\underline{\mathrm b}_v=1$, so $2\geq\delta\geq 1$.
Knowing that $m^{(2)}_{L''}=2$, we deduce from Theorem C that $\dag=1$ or $2$ and $\dag\geq\delta$.




\subsection{The $p>2$ case}\label{su:eat} For $p>2$, consider $A$ over $K=\F_p(t)$ defined by
$$Y^2=X(X-1)(X+t^2).$$
The discriminant of the equation over $\F_p[t]$ is $t^4(t^2+1)^2$ that is minimal \cite[VII.1.1]{sil86}.
Also, take $u=1/t$, then by the change of variable, $A$ is defined by 
$$Y^2=X(X-u^2)(X+1)$$
whose discriminant over $\F_p[s]$ equals $u^4(u^2+1)^2$, is also minimal. Thus, $A$ has semi-stable reduction everywhere. Let $(\xi)_0$ denote the zero of $\xi\in\F_p(t)$. Then 
$$\Delta_{A/K}=4(t)_0+4(u)_0+2(t^2+1)_0.$$
By \cite[Corollary of Proposition 3]{tan93}, the $L$-function $L_{A/K}=1$, so $\theta=0$, and hence by 
\eqref{e:lst}, the $\mu$-invariant over $L/K$ is zero. 

Magma works only when $K'$ is a rational function field, or equivalently $K'/K$ is defined by $T^p-T=\frac{\alpha t+\beta}{\alpha'+\beta' t}$, so $K'/K$ can only ramified at one place $v$ of degree $1$ with $\lambda_v=1$.

\subsubsection{The $p=3$ case}\label{ss:p3} By \cite[Theorem 4.1(a), \S V]{sil86},
there are exactly two supersingular places of $A/K$: $v_+=(t-1)_0$ and $v_{-}=(t+1)_0$.
By \eqref{e:tan22}, we must have $n_{v_+}=n_{v_{-}}=1$. 
Let $w_i$, $i=+,-$, denote the place of $k$ sitting over $v_i$. Since $p-1=2$, $e_{w_i}=1$ or $2$. If $e_{w_i}=1$, then by \eqref{e:env}, $n_{v_i}=n_{w_i}\equiv 0\pmod{2}$, which is a contradiction. Therefore, $e_{w_i}=2$. 

Suppose $K'/K$ ramifies only at $v_+$ and $v_-$.  If $\lambda_{v_+}\leq 1$ and $\lambda_{v_-}\leq 1$,
then by Lemma \ref{l:small} and Lemma \ref{l:deltav}, $\delta=0$, and hence by Theorem C, $\mu_{L'/K'}=0$.
For $K'/K$ defined by $T^3-T=1/(t-1)+\alpha$, or $1/(t+1)+\alpha$, $\alpha\in\F_3$, 
Magma also returns $\mu_{L'/K'}=0$.

If $\lambda_{v_+}=2$, $\lambda_{v_-}\leq 1$, then by calculation $\overline{\mathrm b}_{v_+}=2$,
$\underline{\mathrm b}_{v_+}=1$ 
and
$\overline{\mathrm b}_{v_-}=\underline{\mathrm b}_{v_-}=0$, so $\delta=1$ or $2$. By Theorem C,
$\dag=\delta$ and $4\geq m_{L'}\geq 1$. 
If $\lambda_{v_+}=\lambda_{v_-}=2$, then $4\geq \delta=\dag\geq 2$ and $2\delta\geq m_{L'}\geq \delta$.

\subsubsection{The $p=5$ case}\label{ss:p5} Applying \cite[Theorem 4.1(a), \S V]{sil86}, we find that $v=(t^2+3)_0$ is the only supersingular place of $A/K$. Since $\deg v=2$, we have $n_v=2$, by \eqref{e:tan22}.
Since $n_w=2e_w$ is divisible by $4$ and $e_w\leq 4$, $e_w=2$, or $4$. The extension $k/K$ is unramified outside $v$ and tamely ramified at $v$, so by the class field theory it is a quadratic extension, and hence $e_w=2$.

Again, if $\lambda_v=1$, then $\delta=0=\dag=\mu_{L'}$. 
If $\lambda_v=2$, then $\overline{\mathrm b}_v=2$, $\underline{\mathrm b}_v=1$, so $\delta=1$ or $2$, $\dag=\delta$, and $4\dag\geq m_{L'}\geq \dag$.

\subsection{An example of finite $\mu_{L/K}$ with infinite $\mu_{L'/K'}$}\label{su:non}
Let $F:=\F_2(t)$. Our strategy is to apply \cite{lltt15} by finding an elliptic curve $B/F$, together with disjoint quadratic extensions $K/F$ and $F'/F$, satisfying:
\begin{enumerate}
\item[(a)] $B/F$ is a non-isotrivial semi-stable elliptic curve such that the Hasse-Weil $L$-function 
$L_{B/K'}(s) $ does not vanish at $s=1$. Here $K'=KF'$.
\item[(b)] There is a place $u$ of $F$, at which $B$ has split multiplicative reduction, and there is only one place of $K'$ siting above $u$. 


\end{enumerate}
Having found these, we choose a $\Z_p$-extension $L/K$ as follows.
Let $v$ be the place of $K$ sitting over $u$. By the global class field theory, the maximal pro-$p$
abelian extension $\tilde K$ of $K$, unramified out side $\{v\}$ and 
dihedral over $F$, has Galois group $\tilde\Gamma:=\Gal(\tilde K/K)$ of infinite rank over $\Z_p$.
Here $\tilde K/F$ is dihedral in the sense that if $\tilde \tau\in\Gal(\tilde K/F)$ is a lift of the non-trivial element $\tau\in\Gal(K/F)$, then $\tilde \tau \sigma\tilde\tau^{-1}=\sigma^{-1}$, for every $\sigma\in\tilde\Gamma$.
We choose $L/K$ with $\Gamma=\Gal(L/K)$ a $\Z_p$-quotient of $\tilde\Gamma$ such that
at $v$, $L/K$ is totally ramified and the local field extensions $LK_v/K_v$, $F'K_v/K_v$ are disjoint. The infinite rank of $\tilde\Gamma$ guarantees the existence of such $L$. 

Let $A/F$ be the twist of $B$ by $F'/F$, so that $A/F'=B/F'$, has split-multiplicative reduction at the place
$u'$ sitting over $u$. In particular, the associated $L_{A/K'}(s)=L_{B/K'}(s)$ does not vanish at $s=1$.  By the choice of $L$, $L'/F'$ is dihedral,
at the place $v'$ of $K'$ sitting over $v$, the $\Z_p$-extension
$L'/K'$ is totally ramified and is unramified outside $v'$.
We apply \cite{lltt15} by taking respectively $A$, $K$, $k$, and $L$ in the paper to be our $A$, $K'$, $F'$, and $L'$, and obtain from \S2 of it
the assertion that $X_{L'}$ is non-torsion. Hence $\mu_{L'/K'}=\infty$.

Since $L_{A/K'}(s)=L_{A/K}(s)\cdot L_{B/K}(s)$, the analytic rank of $A/K$ is also $0$, so $\Sel_{p^\infty}(A/K)$
is finite \cite{kt03}. Also, $L/K$ is ramified only at $v$, at which $A/K$ has either additive or non-split multiplicative reduction
\cite[\S C. Theorem 14.1]{sil86}.
\begin{lemma}\label{l:torsion}
$X_L$ is torsion.
\end{lemma}
\begin{proof} Let $\sigma$ be a topological generator of $\Gamma$ and let
$x_1,...,x_n$ be a set of generators of $X_L$ over $\Lambda_\Gamma$. If $X_L$ were non-torsion, then $\Sel_{p^\infty}(A/L)^\Gamma$ would have been of positive $\Z_p$ corank, for other wise, there would be
non-negative integers $a_1,...,a_m$ such that $p^{a_i}x_i\in (\sigma-1)X_L$, for each $i$, and consequently by linear algebra, there is an $n\times n$-matrix $\mathsf M\in \mathrm{Mn}(\Lambda_\Gamma)$, such that modulo the ideal $(\sigma-1)$, $\mathsf M$ is congruent to the diagonal matrix $\mathsf D$ with $\mathsf D_{ii}=p^{a_i}$, and if $X$ is the column vector with $X_i=x_i$, then $\mathsf M\cdot X=0$, so
$\det\mathsf M\not=0$ annihilates $X_L$. That is a contradiction.

Since $A/K$ is non-isotrivial, the group $A_{p^\infty}(L)=p^\nu$ for some finite $\nu$ (see \cite{blv09}),
so by Lemma \ref{l:(1)}, if $\mathrm{res}_{L/K}:\coh^1(K,A_{p^\infty})(p)\longrightarrow \coh^1(L,A_{p^\infty})^\Gamma(p)$
is the restriction map, then the subgroup $\mathsf R:=\mathrm{res}_{L/K}(\mathrm{res}_{L/K}^{-1}(\Sel_{p^\infty}(A/L)^\Gamma))$ in $\Sel_{p^\infty}(A/L)^\Gamma$ is of finite index (see also the proof of Lemma \ref{l:cont}), so $\mathsf R$ must have positive $\Z_p$ corank. 

Consider the exact sequence
$$\xymatrix{0\ar[r] & \Sel_{p^\infty}(A/K) \ar[r] & \mathsf R \ar[r] & \bigoplus_{w} \coh^1(\Gamma_w, A(L_{\tilde w}))(p).}
$$
Here $\tilde w$ denotes a place of $L$ sitting over $K$. $\coh^1(\Gamma_w, A(L_{\tilde w}))(p)$ is independent of the choice of $\tilde w$, and is finite for $w\not=v$ (Lemma \ref{l:unrambound}, as $L/K$ is unramified outside $v$). 
Since $\Sel_{p^\infty}(A/K)$ is finite, the image of $\mathsf R$ in $\coh^1(\Gamma_v, A(L_{\tilde v}))$ must be of positive $\Z_p$ corank. In view of the composition
$$\xymatrix{\coh^1(\Gamma_v, A(L_{\tilde v}))\ar[r]^-\alpha & \coh^1(L'_{\tilde v'}/K_v, A( L'_{\tilde v'}))
\ar[r]^-\beta & \coh^1(L'_{\tilde v'}/K'_{v'}, A( L'_{\tilde v'}))^{G_v},}$$
where $\alpha$ is the inflation, and hence injective, and $\beta$ the restriction, having finite kernel (Lemma \ref{l:F'w'}), we deduce that $ \coh^1(L'_{\tilde v'}/K'_{v'}, A( L'_{\tilde v'}))^{G_v}$ is of positive corank.

Let $Q$ be the local Tate period of $B/K_v$. The exact sequence 
$$\xymatrix{0\ar[r] & Q^\Z\ar[r] & L'^*_{\tilde v'} \ar[r] 
& B(L'_{\tilde v'})\ar[r] & 0}$$
of Tate curve induces the long exact sequence
$$\xymatrix{ \coh^1(L'_{\tilde v'}/K'_{v'},L'^*_{\tilde v'}) \ar[r] & \coh^1(L'_{\tilde v'}/K'_{v'}, B( L'_{\tilde v'})) \ar[r] & \coh^2(L'_{\tilde v'}/K'_{v'}, Q^\Z) 
\ar[r] & \coh^2(L'_{\tilde v'}/K'_{v'},L'^*_{\tilde v'}).}$$
The above is an exact sequence of $\Gal(L'_{\tilde v'}/K_{v})$-modules, with $\Gal(L'_{\tilde v'}/K'_{v'})$ acting trivially
(because the items are cohomology groups of $\Gal(L'_{\tilde v'}/K'_{v'})$), so it is an exact sequence of $G_v$-modules. Hilbert's Theorem $90$ says $\coh^1( L'_{\tilde v'}/K'_{v'},L'^*_{\tilde v'})=0$. Because $Q\in K_v^*$ and $L'_{\tilde v'}/K_{v}$ is an abelian extension, $G_v$ acts trivially on 
 $$\coh^2(L'_{\tilde v'}/K'_{v'}, Q^\Z)=\coh^2(L'_{\tilde v'}/K'_{v'}, \Z)=\Hom(\Gal(L'_{\tilde v'}/K'_{v'}), \Q/\Z),$$
The exact sequence implies that $G_v$ also acts trivially on $\coh^1(L'_{\tilde v'}/K'_{v'}, B( L'_{\tilde v'}))$.
 
As an abelian group $\coh^1(L'_{\tilde v'}/K'_{v'}, A( L'_{\tilde v'}))=\coh^1(L'_{\tilde v'}/K'_{v'}, B( L'_{\tilde v'}))$,
while the action of $G_v$ is twisted by the non-trivial character of $G_v$, the non-trivial element of $G_v$
thus acts on
$\coh^1(L'_{\tilde v'}/K'_{v'}, A( L'_{\tilde v'}))$ as $-1$.
Therefore, $\coh^1(L'_{\tilde v'}/K'_{v'}, A( L'_{\tilde v'}))^{G_v}$ is annihilated by $2$, and hence having trivial $2$-divisible part. This is a contradiction.

\end{proof}

Suppose $B$ is defined by
$$Y^2+XY=X^3+a_2X^2+a_6$$
and $K/F$, $F'/F$ are disjoint quadratic extensions respectively defined by the Artin-Schreier equations
$$T^2-T=\mathsf b,\; T^2-T=\mathsf c.$$
We have
$$L_{A/K'}(s)=\prod_{\chi\in \widehat{\Gal(K'/F)}} L_{A/F}(\chi,s)=L_{A/F}(s)\cdot L_{A_{\mathsf b}/F}(s)\cdot L_{A_{\mathsf c}/F}(s)\cdot L_{A_{\mathsf b+c}/F}(s),$$
where, as in \S\ref{su:ch2}, $B_{\mathsf e}$ denotes the twist curve defined by
$$Y^2+XY=X^3+(a_2+\mathsf e)X^2+a_6.$$

Consider the curve $B$ defined by
\begin{equation}\label{e:5511}
Y^2+XY=X^3+\frac{t^2}{(t+1)^3}X^2+\frac{t^5(t^2+t+1)}{(t+1)^{12}},
\end{equation}
It is semi-stable, having $\Delta=\frac{t^5(t^2+t+1)}{(t+1)^{12}}$, $j=\frac{(t+1)^{12}}{t^5(t^2+t+1)}$, split multiplicative reduction at the zeros of $t$ and $1/t$, non-split multiplicative reduction at the zero $t^2+t+1$. The $L$-function $L_{B/F}(s)=1$.
Choose $\mathsf b=t^5+1/t$, $\mathsf c=1$. Then $K/F$ and $F'/F$ are disjoint from each other.
We have $L_{B_{\mathsf c}/F}(s)=1$, $L_{B_{\mathsf b}/F}(1)=1/4$ $L_{B_{\mathsf b+c}/F}(1)=4$, hence (a) holds.
Let $v$ be the zero of $t$ or $1/t$. Since $K/F$ is ramified at $v$ and $F'=\F_4(t)$ which is is inert at $v$, (b) also holds. Therefore, by Lemma \ref{l:torsion}, $\mu_L$ is finite, while $\mu_{L'/K'}=\infty$.


\end{document}